\documentclass[11pt]{article}
\usepackage{amssymb}
\usepackage{amsbsy}
\usepackage[latin1]{inputenc}
\usepackage{amsthm}
\usepackage[dvips]{graphicx}
\usepackage{graphicx} 
\usepackage{subfigure}
\usepackage{pst-eucl}
\usepackage[latin1]{inputenc}
\usepackage[english]{babel}
\usepackage{amsmath,amssymb,graphics,mathrsfs}
\usepackage{amsmath,amssymb,latexsym}
\usepackage{graphicx,color}
\usepackage[T1]{fontenc}
\usepackage[active]{srcltx}
\usepackage{multicol}
\usepackage[latin1]{inputenc}
\usepackage{pst-all}
\usepackage{enumerate}
\usepackage{pstricks}
\usepackage{pstricks-add}
\usepackage{setspace}
\usepackage{soul}
\usepackage{cancel}
\usepackage{nonfloat}
\usepackage[margin=10pt,font=footnotesize,labelfont=bf,labelsep=endash]{caption}
\usepackage[left=4cm,top=3cm,right=2.4cm,bottom=3cm]{geometry}
\usepackage[colorlinks=true,citecolor=red,linkcolor=blue,urlcolor=RubineRed,pdfpagetransition=Blinds,pdftoolbar=false,pdfmenubar=false]{hyperref}

\newcommand{\R}{\hbox{\rm I \kern -5pt R}}     % simbolo de Reales
\newcommand{\p} {\hbox{\rm I \kern -5pt P}}
\def\x  {\boldsymbol x}
\def\H        {{\boldsymbol H}}

\newtheorem{prop}{Proposition}[section]
\newtheorem{defi}[prop]{Definition}%[chapter]
\newtheorem{tma}[prop]{Theorem}%[chapter]
\newtheorem{cor}[prop]{Corollary}%[chapter]
\newtheorem{obs}[prop]{Remark}%[chapter]
\newtheorem{lem}[prop]{Lemma}%[chapter]
\allowdisplaybreaks

\begin{document}

\title{Comparison of two finite element schemes for a chemo-repulsion system with quadratic production}
\author{F.~Guill\'en-Gonz\'alez\thanks{Dpto. Ecuaciones Diferenciales y An\'alisis Num\'erico and IMUS, 
Universidad de Sevilla, Facultad de Matem\'aticas, C/ Tarfia, S/N, 41012 Sevilla (SPAIN). Email: guillen@us.es, angeles@us.es},
M.~A.~Rodr\'\i guez-Bellido$^*$~and
 D.~A.~Rueda-G\'omez$^*$\thanks{Escuela de Matem\'aticas, Universidad Industrial de Santander, A.A. 678, Bucaramanga (COLOMBIA). Email:  diaruego@uis.edu.co}}

\date{}
\maketitle

\begin{abstract}
In this paper we propose two fully discrete Finite Elements (FE) schemes for a repulsive chemotaxis model with quadratic production term. The first one (called scheme \textbf{UV}) corresponds to the backward Euler in time with FE in space approximation; while the second one (called scheme \textbf{US}$_\varepsilon$) is obtained as a modification of the scheme \textbf{US} proposed in \cite{FMD2}, by applying a regularization procedure. We prove that the schemes \textbf{UV} and \textbf{US}$_\varepsilon$ have better properties than the FE scheme \textbf{US}. Specifically, we prove that, unlike the scheme \textbf{US}, the scheme \textbf{UV} is energy-stable in the primitive variables of the model, under a ``compatibility'' condition on the FE spaces. On the other hand, the scheme \textbf{US}$_\varepsilon$ is energy-stable with respect to the same modified energy of the scheme \textbf{US}, and  an ``approximated positivity'' property holds (which is not possible to prove for the schemes \textbf{US} and \textbf{UV}). Additionally, we study the well-posedness of the schemes and the long time behaviour obtaining exponential convergence to constant states. Finally, we compare the numerical schemes throughout several numerical simulations.
\end{abstract}



\section{Introduction}
The directed movement of cells in response to a chemical stimulus is known in biology as chemotaxis. More specifically, if the cells move towards regions of high chemical concentration, the motion is called chemoattraction, while if the cells move towards regions of lower chemical concentration, the motion is called chemorepulsion. Models for chemotaxis motion has been studied in literature (see \cite{Cristian,HP,Horst,KS,Lan1,Lan2,Win1,Win2} and references therein). One of the most important characteristics of chemoattractant models is that the finite blow up of solutions can happen in space dimension greater or equal to $2$; while in chemorepulsion models this phenomenon is not expected. Many works have been devoted to study in what cases and how blow up takes place (see for instance \cite{Win1b,Lin,Via,Win1c,Win1a,XIXI,Zhao}). \\

In those cases in which blow-up phenomenon does not happen, it is interesting to study the asymptotic behaviour of the solutions of the model. In fact, in \cite{Osaki}, Osaki and Yagi studied the convergence of the solution of the Keller-Segel model to a stationary solution in the one-dimensional case. In \cite{HJ}, the convergence of the solution of the Keller-Segel model with an additional  cross-diffusion term to a steady state was shown. In \cite{Cristian} the authors proved the convergence to constant state for a chemorepulsion model with linear production. Therefore, taking into account the results above, the first aim of this paper is to study the asymptotic behaviour of the following parabolic-parabolic repulsive-productive chemotaxis model (with quadratic signal production):
\begin{equation}  \label{modelf00}
\left\{
\begin{array}
[c]{lll}%
\partial_t u - \Delta u = \nabla\cdot (u\nabla v)\ \ \mbox{in}\ \Omega,\ t>0,\\
\partial_t v - \Delta v + v =  u^{2} \ \mbox{in}\ \ \Omega,\ t>0,\\
\displaystyle \frac{\partial u}{\partial \mathbf{n}}=\frac{\partial v}{\partial \mathbf{n}}=0\ \ \mbox{on}\ \partial\Omega,\ t>0,\\
u({\x},0)=u_0({\x})\geq 0,\ v({\x},0)=v_0({\x})\geq 0\ \ \mbox{in}\ \Omega,
\end{array}
\right. \end{equation}
where $\Omega$ is a $n-$dimensional open bounded domain, $n=1,2,3$, with boundary $\partial \Omega$; and the unknowns  are $u(\x, t) \geq 0$, the cell density, and $v(\x, t) \geq 0$, the chemical concentration. This model has been studied in \cite{FMD}. There, the authors shown that model (\ref{modelf00}) is well-posed: there exists global in time weak-strong solution (in the sense of Definition \ref{ws00} below) and, for $1D$ or $2D$ domains, there exists a unique global in time regular solution.\\

On the other hand, another interesting topic is the study of fully discrete FE schemes approximating model (\ref{modelf00}), conserving properties of the continuous problem such as: mass-conservation, energy-stability, positivity and long time behaviour. In fact, in \cite{FMD2} it was studied a fully discrete FE scheme for model (\ref{modelf00}), which is mass-conservative and energy-stable with respect to a modified energy given in terms of the auxiliary variable ${\boldsymbol{\sigma}}=\nabla v$. However, neither energy-stability with respect to the primitive variables $(u,v)$ (see (\ref{eneruva}) below) nor positivity (or approximated positivity) were proved. Moreover, as far as we know, there are not another works studying FE aproximations for problem (\ref{modelf00}). 
For this reason, the second aim of this paper is to present two new fully discrete FE schemes, that have better properties than the scheme proposed in \cite{FMD2}, in terms of energy-stability, positivity and asymptotic behaviour of the scheme.\\

The asymptotic behaviour of fully discrete numerical schemes has been studied in different contexts. In fact, in \cite{GS} Guill\'en-Gonz\'alez and Samsidy proved asymptotic convergence for a fully discrete FE scheme for a Ginzburg-Landau model for nematic liquid crystal flow.  In \cite{MP} Merlet and Pierre studied the asymptotic behaviour of the backward Euler scheme applied to gradient flows. It is important to notice that, in chemotaxis models, there are few works studying large-time behaviour for fully discrete schemes. We refer to \cite{BJ}, where the authors shown conditional stability and convergence at infinite time of a finite volume scheme for a Keller-Segel model with an additional  cross-diffusion term. Meanwhile, the behavior at infinite time of a fully discrete scheme for model (\ref{modelf00}) seems to be still an open problem. \\

Likewise, the energy-stability property has been studied for fully discrete numerical schemes in the chemotaxis framework. In \cite{GRR1}, the authors studied unconditionally energy stables  FE schemes for a chemo-repulsion model with linear production. A finite volume scheme for a Keller-Segel model with an additional cross-diffusion term satisfying the energy-stablity property (conditionally) has been studied in \cite{BJ}. In \cite{FMD2}, it was studied an unconditionally energy-stable FE scheme for model (\ref{modelf00}) with respect to a modified energy written in terms of the auxiliary variable ${\boldsymbol{\sigma}}=\nabla v$. However, up our knowledge, the energy-stability in FE schemes, with respect to the $(u,v)$-energy given in (\ref{eneruva}) below, is so far an open problem.       \\

In terms of positive or approximately positive numerical schemes on chemotaxis context we refer to \cite{Saad1,Zuhr,SG1,FMD,GRR1}. In \cite{Zuhr}, the nonnegativity of numerical methods, using FE techniques, to a generalized Keller-Segel model was analyzed. A discrete maximum principle for a fully discrete numerical scheme (combining the finite volume method and the nonconforming finite element
method) approaching a chemotaxis-swimming bacteria model was obtained in \cite{Saad1}. In \cite{GRR1}, aproximated positivity of FE schemes for a chemo-repulsion model with linear production was proved. The positivity of a finite volume scheme for a parabolic-elliptic chemotaxis system was studied in \cite{SG1}. In \cite{FMD}, positivity of only time-discrete schemes associated to model (\ref{modelf00}) was proved.
However, there are not works studying positive (or approximately positive) FE schemes for model (\ref{modelf00}).
 \\

Consequently, the main novelties in this paper are the following:
\begin{itemize}
	\item The introduction of a FE scheme (see scheme \textbf{UV} in Section \ref{NEs} below) which is energy-stable with respect to the $(u,v)$-energy of the continuous problem (\ref{modelf00}) given in (\ref{eneruva}), under a ``compatibility'' condition on the FE spaces, namely taking   $(\mathbb{P}_m,\mathbb{P}_{2m})$-continuous FE (with $m\geq 1$) for $(u,v)$.
	\item The introduction of another FE scheme (see scheme \textbf{US}$_\varepsilon$ in Section \ref{schemeUSe} below) which has the ``approximated'' positivity property, and it is energy-stable (with respect to a modified $(u,{\boldsymbol{\sigma}})$-energy).
	\item The proof of the long time behaviour for the schemes previously mentioned, obtaining  exponential convergence to constant states as time goes to infinity. 
\end{itemize}

The outline of this paper is as follows: In Section \ref{CMo}, we study (formally) the asymptotic behavior of the global solutions for the model (\ref{modelf00}), and we prove the exponential convergence as time goes to infinity to constant states. In Section \ref{NEs}, we study a fully discrete scheme associated to model (\ref{modelf00}), corresponding to the nonlinear backward Euler in time and FE in space in the variables $(u,v)$. The analysis includes the well-posedness of the scheme and some properties such as $u$-conservation, energy stability, convergence and long time behaviour. In Section \ref{schemeUSe}, we propose another fully discrete FE approximation of model (\ref{modelf00}), which is obtained combining the scheme \textbf{US} proposed in \cite{FMD2} with a regularization technique. For this scheme, we can prove, in addition to the properties proved for the previous scheme, the approximated positivity. In Section \ref{NSi}, we compare the numerical schemes throughout several numerical simulations, giving the main conclusions in Section \ref{Con}.

\subsection{Notation}
We recall some functional spaces which will be used throughout this paper. We will consider the usual Sobolev spaces $H^m(\Omega)$ and Lebesgue spaces $L^p(\Omega),$
$1\leq p\leq \infty,$ with norms $\Vert\cdot\Vert_{m}$ and $\Vert\cdot \Vert_{L^p}$, respectively. In particular,  the $L^2(\Omega)$-norm will be denoted by $\Vert
\cdot\Vert_0$. We denote by $\H^{1}_{\sigma}(\Omega):=\{\mathbf{u}\in \H^{1}(\Omega): \mathbf{u}\cdot \mathbf{n}=0 \mbox{ on } \partial\Omega\}$ and we will use the following equivalent norms in $H^1(\Omega)$  and ${\bf H}_{\sigma}^1(\Omega)$, respectively (see \cite{necas} and \cite[Corollary 3.5]{Nour}, respectively):
\begin{equation}
\Vert u \Vert_{1}^2=\Vert \nabla u\Vert_{0}^2 + \left( \int_\Omega u\right)^2, \ \ \forall u\in H^1(\Omega),
\end{equation}
\begin{equation}\label{H1div}
\Vert {\boldsymbol\sigma} \Vert_{1}^2=\Vert {\boldsymbol\sigma}\Vert_{0}^2 + \Vert \mbox{rot }{\boldsymbol\sigma}\Vert_0^2 + \Vert \nabla \cdot {\boldsymbol\sigma}\Vert_0^2, \ \ \forall {\boldsymbol\sigma}\in \H^{1}_{\sigma}(\Omega).
\end{equation}
In particular, (\ref{H1div}) implies that
\begin{equation*}
\Vert\nabla v\Vert_{1}^2=\Vert \nabla v\Vert_{0}^2  + \Vert \Delta v\Vert_0^2, \ \ \forall v:\nabla v\in \H^{1}_{\sigma}(\Omega).
\end{equation*}
If $Z$ is a
general Banach space, its topological dual will be denoted by $Z'$.
Moreover,the letters  $C,C_i,K_i$ will denote different positive constants depending on the data $(\Omega,u_0, v_0)$, 
but independent of the discrete parameters $(k, h)$ and time step $n$, which may change from line to line (or even within the same line). 

\section{Continuous problem}\label{CMo}
In this section some fundamental concepts associated to problem (\ref{modelf00}) are presented, including the definition of weak-strong solutions and some qualitative properties such as $u$-conservation, positivity and large time behaviour. In particular, exponential convergence to constant states as time goes to infinity is obtained.

\subsection{Some properties}
Problem (\ref{modelf00}) conserves in time the total mass $\int_\Omega u$. In fact, defining 
\begin{equation}\label{m0}
	m_0=\frac1{|\Omega|} \int_{\Omega} u_0,
	\end{equation}
 and  integrating (\ref{modelf00})$_1$ in $\Omega$,
\begin{equation*}
\frac{d}{dt}\left(\int_\Omega u\right)=0, \ \ \mbox{ i.e. } \ 
\int_\Omega u(t)=\int_\Omega u_0:= m_0\vert \Omega\vert, \ \ \forall t>0.
\end{equation*}
%Moreover, integrating (\ref{modelf00})$_2$ in $\Omega$ one can deduce the following behavior of $\int_\Omega v$:
%\begin{equation*}
%\frac{d}{dt}\left(\int_\Omega v\right) + \int_\Omega
%v=\int_\Omega u^2 .
%\end{equation*}
Now, the definition of weak-strong solutions for problem (\ref{modelf00}) is presented.
\begin{defi} \label{ws00}{\bf (Weak-strong solutions of (\ref{modelf00}))} 
	Given $(u_0, v_0)\in L^2(\Omega)\times H^1(\Omega)$ with $u_0\geq 0$, $v_0\geq 0$ a.e.~$\x\in \Omega$.	A pair $(u,v)$ is called weak-strong solution of problem (\ref{modelf00}) in $(0,+\infty)$, if $u\geq 0$, $v\geq 0$ a.e.~$(t,\x)\in (0,+\infty)\times \Omega$,
	\begin{equation}\label{wsa}
	(u-m_0,v-m_0^2) \in L^{\infty}(0,+\infty;L^2(\Omega)\times H^1(\Omega)) 
	\cap L^{2}(0,+\infty;H^1(\Omega)\times H^2(\Omega)),  %\ \ \forall T>0,
	\end{equation}
	\begin{equation}\label{wsa-bis}
	(\partial_t u, \partial_t v) \in L^{q'}(0,T;H^1(\Omega)' \times L^2(\Omega)), \ \ \forall T>0,
	\end{equation}
	where $q'=2$ in the $2$-dimensional case $(2D)$ and $q'=4/3$ in the $3$-dimensional case $(3D)$ ($q'$ is the conjugate exponent of $q=2$ in $2D$ and $q=4$ in $3D$); the following variational formulation holds
	\begin{equation}\label{wf01}
	\int_0^T \langle \partial_t u,\overline{u}\rangle + \int_0^T (\nabla u,  \nabla \overline{u}) +\int_0^T (u\nabla v,\nabla \overline{u})=0, \ \ \forall \overline{u}\in L^q(0,T;H^{1}(\Omega)), \ \ \forall T>0,
	\end{equation}
	the following equation holds pointwisely
	\begin{equation}\label{wf02}
	\partial_t v +A v=u^2 \ \ \mbox{ a.e. } (t,\x)\in (0,+\infty)\times\Omega,
	\end{equation}
	the initial conditions $(\ref{modelf00})_4$ are satisfied and the following energy inequality (in integral version) holds  a.e.~$t_0,t_1$ with $t_1\geq t_0\geq 0$:
	\begin{equation}\label{wsd}
	\mathcal{E}(u(t_1),v(t_1)) - \mathcal{E}(u(t_0),v(t_0))
	+ \int_{t_0}^{t_1} \left(\Vert \nabla u(s)  \Vert_{0}^2 
	+\frac{1}{2} \Vert \nabla v(s) \Vert_{1}^2
	% +\frac{1}{2} \Vert \Delta v(s) \Vert_{0}^2
	\right)\ ds \leq0,
	\end{equation}
	where
	\begin{equation}\label{eneruva}
	\mathcal{E}(u,v)=\displaystyle
	\frac{1}{2}\Vert u\Vert_{0}^2 + \frac{1}{4}\Vert \nabla v\Vert_{0}^{2}.
	\end{equation}
\end{defi}
\begin{obs}
In particular, the energy inequality (\ref{wsd}) is valid for $t_0=0$. Moreover, (\ref{wsd}) shows the dissipative character of the model with respect to the total energy $\mathcal{E}(u(t), v(t))$. 
\end{obs}
\begin{obs} {\bf (Positivity)}\label{OBSP}
$u\geq 0$ in $1D$ and $2D$ domains and $v\geq 0$ in any ($1D$, $2D$ or $3D$) dimension are a consequence of (\ref{wsa})-(\ref{wf02}). Indeed, this follows from the fact that in these cases we can test (\ref{wf01}) by $u_{-}:= \min\{u,0\} \in L^{2}(0,T;H^1(\Omega))$ and (\ref{wf02}) by $v_{-}:= \min\{v,0\}\in L^{2}(0,T;H^2(\Omega))\hookrightarrow L^{2}(0,T;L^2(\Omega))$. Notice that in 3D domains, $u_{-}$ has no the sufficient regularity in order to take it as test function. Hence the positivity of $u$ cannot be deduced from (\ref{wsa})-(\ref{wf01}), which must be explicitly imposed.
\end{obs}

In \cite{FMD}, it was proved the existence of weak-strong solutions of problem (\ref{modelf00}) (satisfying in particular the energy inequality (\ref{wsd})), through convergence of a time-discrete numerical scheme associated to model (\ref{modelf00}). Hereafter, in order to abbreviate, we will use the following notation: 
$$\hat{u}:= u - m_0, \ \ \hat{v}=v- m_0^2$$
for $m_0$ defined in (\ref{m0}).

\subsection{Convergence at infinite time}\label{subcc}
In this subsection, the asymptotic analysis of problem (\ref{modelf00}) is going to be analyzed in a formal manner, without justifying the computations and assuming sufficient regularity for the exact solution $(u,v)$. Our main interest is to reproduce the long time behaviour in fully discrete numerical schemes.

\

First, we define:
$$E(t):= \Vert \hat{u}(t)\Vert_0^2 + \frac{1}{2}\Vert \nabla v(t)\Vert_0^2 \  \mbox{ and } \ F(t):= \Vert \nabla \hat{u}(t)\Vert_0^2 + \frac{1}{2}\Vert \nabla v(t)\Vert_1^2.$$ Then, taking $\bar{u}=\hat{u}$ in (\ref{wf01})  and testing (\ref{wf02}) by $\bar{v}=-\frac{1}{2}\Delta v$, one arrives at
\begin{equation}\label{weakcont2}
\frac{1}{2} E'(t)+ F(t) =0.
\end{equation}
Therefore, using the Poincar\'e inequality $\Vert \nabla \hat{u}\Vert_0^2\geq C_p \Vert \hat{u}\Vert_0^2$ one has that $ 2F(t)\geq 2(C_p \Vert \hat{u}(t)\Vert_0^2 + \frac{1}{2}\Vert \nabla v(t)\Vert_1^2)\geq 2K_p E(t)$ (with $K_p=\min\{C_p,1\}$), and from (\ref{weakcont2}) one can deduce 
\begin{equation}\label{c21}
E(t)\leq  \Vert (\hat{u}_0,\nabla v_0)\Vert_0^2e^{-2K_p t}, \ \ \forall t\geq 0.
\end{equation}
Moreover, testing (\ref{wf02}) by $\bar{v}=\hat{v}$ and using (\ref{wsa}) and (\ref{c21}), one has 
$$
\frac{d}{dt}\Vert \hat{v}\Vert_0^2 + \Vert \hat{v}\Vert_1^2 \leq C\Vert \hat{u} \Vert_{0}^2\Vert \hat{u} + 2m_0\Vert_{L^3}^2\leq C e^{-2K_pt}(1+\Vert \hat{u}\Vert_1^2),
$$
 from which one arrives at
\begin{equation}\label{newaa}
\Vert \hat{v}(t)\Vert_0^2 \leq \Vert \hat{v}_0\Vert_0^2 e^{-t}  + C e^{-t}  \int_{0}^{t} e^{(1-2K_p)s} \, ds + C e^{-t}  \int_{0}^{t} e^{(1-2K_p)s} \Vert \hat{u}(s)\Vert_1^2 \, ds.
\end{equation}
The last two terms on the right hand side of (\ref{newaa}) are bounded by
\begin{equation}\label{B001} 
C e^{-t}  \int_{0}^{t} e^{(1-2K_p)s} \, ds\leq 
\left\{\begin{array}{l}
C e^{-t} \ \mbox{ if } 2K_p>1,\\
Cte^{-t}  \ \mbox{ if } 2K_p=1,\\
C e^{-2K_pt}  \ \mbox{ if } 2K_p<1,
\end{array}\right.
\end{equation} 
and
\begin{equation}\label{B001-a} 
C e^{-t}  \int_{0}^{t} e^{(1-2K_p)s} \Vert \hat{u}(s)\Vert_1^2 \, ds\leq 
\left\{\begin{array}{l}
C e^{-t} \ \mbox{ if } 2K_p>1,\\
Ce^{-t}  \ \mbox{ if } 2K_p=1,\\
C e^{-2K_pt}  \ \mbox{ if } 2K_p<1,
\end{array}\right.
\end{equation} 
where (\ref{wsa}) was used in (\ref{B001-a}). Thus, from (\ref{newaa})-(\ref{B001-a}) one can deduce that, for any t>1,
\begin{equation*}
\Vert \hat{v}(t)\Vert_0^2 \leq \Vert \hat{v}_0\Vert_0^2 e^{-t}  + 
\left\{\begin{array}{l}
C e^{-t} \ \mbox{ if } 2K_p>1,\\
Cte^{-t}  \ \mbox{ if } 2K_p=1,\\
C e^{-2K_pt}  \ \mbox{ if } 2K_p<1,
\end{array}\right.
\leq 
C\left\{\begin{array}{l}
 e^{-t} \ \mbox{ if } 2K_p>1,\\
te^{-t}  \ \mbox{ if } 2K_p=1,\\
 e^{-2K_pt}  \ \mbox{ if } 2K_p<1.
\end{array}\right.
\end{equation*} 
\

\section{Scheme \textbf{UV}}\label{NEs}
The first scheme that will be studied in this paper is obtained by using FE in space and backward Euler in time for the system (\ref{modelf00}) 
(considered for simplicity on a uniform partition of $[0,+\infty)$ given by $t_n=nk$, where $k>0$ denotes the time step). Concerning the space discretization, we consider a family of shape-regular and quasi-uniform triangulations $\{\mathcal{T}_h\}_{h>0}$  of $\overline{\Omega}$ made up of simplexes 
(intervals in one dimension,
triangles in two dimensions and tetrahedra in three dimensions), so that $\overline{\Omega}= \cup_{K\in \mathcal{T}_h} K$, where $h = \max_{K\in \mathcal{T}_h} h_K$, with $h_K$ being the diameter of $K$. 
%Further, let $\mathcal{N}_h = \{a_i\}_{i\in \mathcal{I}}$ denote the set of all the nodes of $\mathcal{T}_h$. 
We choose FE spaces for $u$ and $v$, which we denote by 
\begin{equation*}
(U_h, V_h) \subset H^1 \times W^{1,6} \mbox{ generated by } (\mathbb{P}_m,\mathbb{P}_{2m})\mbox{-continuous FE, with } m\geq 1.
\end{equation*} 
With this choice, $(u_h^n)^2 \in V_h$ is guaranteed, which will be the key point to prove the energy stability of this scheme (see Lemma \ref{estinc1uv} below). Then, the following first order in time, nonlinear and coupled scheme is considered (hereafter, we denote $\delta_t a^n= (a^n - a^{n-1})/k$):
\begin{itemize}
	\item{\underline{\emph{Scheme \textbf{UV}:}}\\
		{\bf Initialization}: Let $(u^{0}_h,v^0_h)\in  U_h\times V_h$ be a suitable approximation of $(u_0,v_0)\in L^2(\Omega) \times H^1(\Omega)$, as $h\rightarrow 0$, with $\displaystyle\frac{1}{\vert\Omega\vert}\int_\Omega u^0_h = \displaystyle\frac{1}{\vert\Omega\vert}\int_\Omega u_0 = m_0$. \\
		{\bf Time step} n: Given $(u^{n-1}_h,v^{n-1}_h)\in  U_h\times V_h$, compute $(u^{n}_h,v^{n}_h)\in  U_h\times V_h$ solving
		\begin{equation}
		\left\{
		\begin{array}
		[c]{lll}%
		(\delta_t u^n_h,\bar{u}_h) + (\nabla u^n_h, \nabla \bar{u}_h) +(u^n_h\nabla v^n_h,\nabla \bar{u}_h)=0, \ \ \forall \bar{u}_h\in U_h,\\
		(\delta_t v^n_h,\bar{v}_h)  +(\nabla v^n_h, \nabla \bar{v}_h) + (v^n_h,\bar{v}_h)  -((u^n_h)^2,\bar{v}_h) = 0, \ \ \forall
		\bar{v}_h\in V_h.
		\end{array}
		\right.  \label{modelf02uv}
		\end{equation}
}
\end{itemize}

\subsection{Mass-conservation, well-posedness, energy-stability and convergence}\label{EESuv}
In this subsection, we follow the arguments presented in \cite{FMD2}. Consequently, the results will be presented omiting technical details. Assuming that $1\in U_h$ and $1\in V_h$, the scheme \textbf{UV} satisfies
\begin{equation}\label{consuuv}
\int_\Omega u^n_h=\int_\Omega u^{n-1}_h=\cdot\cdot\cdot=\int_\Omega
u^{0}_h=m_0\vert \Omega\vert,
\end{equation}
and
\begin{equation}\label{compv1uv} 
\delta_t \left(\int_\Omega v^n_h \right)= \int_\Omega (u^n_h)^2  - \int_\Omega v^n_h.
\end{equation}
\begin{tma} {\bf(Unconditional solvability and conditional uniqueness)}
There exists \linebreak{$(u^n_h,v^n_h) \in  U_h\times V_h$} solution of the scheme \textbf{UV}. Moreover, if 
\begin{equation*}
k\Vert (u^n_h,\nabla v^n_h)
\Vert_{1}^4 \quad \hbox{is small enough,}
\end{equation*}
then the solution is unique.
\end{tma}
\begin{proof}
The proof follows the arguments of Theorem 4.4 of \cite{FMD}.
\end{proof}

Let $A_h: H^1(\Omega) \rightarrow V_h$ be the linear operator defined as follows
\begin{equation}\label{discVh}
(A_h v_h, \bar{v}_h)=(\nabla v_h,\nabla \bar{v}_h)+( v_h, \bar{v}_h) , \ \ \forall \bar{v}_h\in V_h.
\end{equation}
Then, the discrete chemical equation (\ref{modelf02uv})$_2$ can be rewritten as
\begin{equation}\label{refvh}
(\delta_t v^n_h,\bar{v}_h)  +(A_h v^n_h, \bar{v}_h) -((u^n_h)^2,\bar{v}_h) = 0, \ \ \forall
\bar{v}_h\in V_h,
\end{equation}
and the following estimate holds (see for instance, Lemma 3.1 in \cite{FMD2}):
\begin{equation}\label{Ah1}
\Vert v_h \Vert_{W^{1,6}}\leq C \Vert A_h v_h\Vert_0, \ \ \forall v_h\in V_h.
\end{equation}

\begin{defi}\label{enesf00}
A numerical scheme with solution $(u^n_h,v^n_h)$ is called energy-stable  if the energy defined in (\ref{eneruva}) is time decreasing, that is, 
	\begin{equation*}
	\mathcal{E}(u^n_h,v^n_h)\leq \mathcal{E}(u^{n-1}_h,v^{n-1}_h), \ \ \forall n\geq 1.
	\end{equation*}
\end{defi}

\begin{lem} {\bf (Unconditional stability)} \label{estinc1uv}
If $(u_h^n,v_h^n)$ is generated by $(\mathbb{P}_m,\mathbb{P}_{2m})$-continuous FE, then the scheme \textbf{UV} is unconditionally energy-stable. In fact, if $(u^n_h,v^n_h)$ is any solution of the scheme \textbf{UV}, then the following discrete energy law holds
\begin{eqnarray}\label{lawenerfydisceuv}
&\delta_t \mathcal{E}(\hat{u}^n_h,v^n_h)&\!\!\!\!\!+ 
\frac{k}{2} 
\Vert \delta_t \hat{u}^n_h\Vert_{0}^2 +
\frac{k}{4} \Vert \delta_t \nabla v^n_h\Vert_{0}^2 
 + \Vert  \hat{u}^n_h\Vert_{1}^{2} +
\displaystyle\frac{1}{2}\Vert (A_h-I) v^n_h\Vert_{0}^{2} +
\displaystyle\frac{1}{2}\Vert \nabla v^n_h\Vert_{0}^{2}=0.
\end{eqnarray}
%where $\mathcal{E}(u^n_h,v^n_h)=\displaystyle
%\frac{1}{2}\Vert u^n_h\Vert_{0}^2 + \frac{1}{4}\Vert \nabla v^n_h\Vert_{0}^{2}$.
\end{lem}
\begin{proof}
Taking $\bar{u}_h=\hat{u}^n_h$ in (\ref{modelf02uv})$_1$, $\bar{v}_h= \displaystyle\frac{1}{2}(A_h -I) v^n_h$
in (\ref{refvh}) and using (\ref{discVh}), (\ref{lawenerfydisceuv}) is deduced.
%\begin{equation}\label{modelfle03uv}
%\int_\Omega u^n_h \cdot \delta_t u^n_h + \Vert \nabla
%u^n_h\Vert_{0}^{2} + \displaystyle\frac{1}{2}\int_\Omega \nabla v^n_h\cdot
%\delta_t \nabla v^n_h +\displaystyle\frac{1}{2}\Vert (A_h-I) v^n_h\Vert_{0}^{2} +
%\displaystyle\frac{1}{2}\Vert \nabla v^n_h\Vert_{0}^{2} =0.
%\end{equation}
%To get (\ref{modelfle03uv}), the fact that $(u^n_h)^2 \in V_h$ is essential (which holds from the choice $(\mathbb{P}_m,\mathbb{P}_{2m})$ appro\-xi\-mation for $(U_h,V_h)$) in order to cancel the terms $(u^n_h\nabla v^n_h,\nabla u^n_h)=\frac{1}{2}(\nabla(u^n_h)^2,\nabla v^n_h)$ and $-\frac{1}{2}((u^n_h)^2,(A_h -I) v^n_h)$. Moreover, using the formula
%$a(a-b)=\displaystyle\frac{1}{2}(a^2-b^2)+\displaystyle\frac{1}{2}(a-b)^2$
%we deduce that
%\begin{equation}\label{modelfle04uv}
%\displaystyle\int_\Omega u^n_h \cdot \delta_t u^n_h + \frac{1}{2}\int_\Omega \nabla v^n_h \cdot
%\delta_t \nabla v^n_h = \delta_t \left(\frac{1}{2}
% \Vert u^n_h\Vert_{0}^2 + \frac{1}{4}  \Vert \nabla v^n_h\Vert_{0}^2 \right) + \frac{k}{2} 
%\Vert \delta_t u^n_h\Vert_{0}^2+ \frac{k}{4} \Vert \delta_t
%\nabla v^n_h\Vert_{0}^2.
%\end{equation}
%Thus, from (\ref{modelfle03uv})-(\ref{modelfle04uv}), we deduce (\ref{lawenerfydisceuv}).
\end{proof}

From the (local in time) discrete energy law (\ref{lawenerfydisceuv}), we deduce the following global in time estimates. 

\begin{lem}  {\bf(Uniform weak-strong estimates)}
Let $(u^n_h,v^n_h)$ be any solution of the scheme \textbf{UV}. Then, the following estimate holds
\begin{equation}\label{weak01uv}
\Vert (\hat{u}^n_h, v^n_h)\Vert_{0\times 1}^{2}
%+k^2 \underset{m=1}{\overset{n}{\sum}}\Vert (\delta_t u^m_h,\delta_t \nabla v^m_h) \Vert_0^2
+ k \underset{m=1}{\overset{n}{\sum}}\left(\Vert \hat{u}^m_h \Vert_{1}^2 +  \Vert  \hat{v}^m_h \Vert_{W^{1,6}}^2\right)  \leq C_0, \ \ \ \forall n\geq 1.
\end{equation}
\end{lem}
\begin{proof}
Multiplying (\ref{lawenerfydisceuv}) by $k$ and summing, one obtains
\begin{equation}\label{weak01uv-a}
\Vert (\hat{u}^n_h, \nabla v^n_h)\Vert_{0}^{2}
+ k \underset{m=1}{\overset{n}{\sum}}\left(\Vert \hat{u}^m_h \Vert_{1}^2 + \Vert \nabla v^m_h \Vert_{0}^2 + \Vert (A_h-I) v^m_h \Vert_{0}^2\right)  \leq C_0, \ \ \ \forall n\geq 1.
\end{equation}
On the other hand, rewriting (\ref{modelf02uv}) as
\begin{equation}\label{vm0}
(\delta_t \hat{v}^n_h,\bar{v}_h)  +( {A}_h \hat {v}^n_h,\bar{v}_h)  =((\hat{u}^n_h+2 m_0)\hat{u}^n_h,\bar{v}_h), \ \ \forall
\bar{v}_h\in V_h,
\end{equation}
and taking $\bar{v}=\hat{v}^n_h$ one has
\begin{eqnarray*}
	&\displaystyle\delta_t \Vert \hat{v}^n_h \Vert_0^2 &\!\!\!\!  
	%+ \frac{1}{2k}\Vert  \hat{v}^n_h - \hat{v}^{n-1}_h \Vert_0^2 
	+ \Vert \hat{v}^n_h\Vert_1^2\leq C\Vert \hat{u}^n_h + 2 m_0\Vert_{L^{3/2}}^2 \Vert \hat{u}^n_h\Vert_{L^6}^2\leq C \Vert \hat{u}^n_h\Vert_{H^1}^2,
\end{eqnarray*}
from which, multiplying by $k$, adding and using (\ref{weak01uv-a}), one can deduce 
\begin{equation}\label{weak02UVlinvL2-a}
\Vert v^n_h \Vert_0^2 + k \underset{m=1}{\overset{n}{\sum}}\Vert \hat{v}^m_h \Vert_1^2  \leq K_0, \ \ \ \forall n\geq 1.
\end{equation}
Then, adding (\ref{weak01uv-a}) and (\ref{weak02UVlinvL2-a}) and using (\ref{Ah1}), (\ref{weak01uv}) is obtained.

%we obtain (\ref{weak01uv}). On the other hand, from (\ref{compv1uv}) and using (\ref{weak01uv}), we have
%\begin{equation}\label{vmed1a}
%(1+k) \left \vert\int_\Omega v^n_h \right\vert - \left\vert \int_\Omega v^{n-1}_h \right\vert\leq k \left\vert \int_\Omega (u^n_h)^2\right\vert = k  \Vert u^n_h \Vert_0^2\leq kC_0.
%\end{equation}
%Then, using Lemma \ref{tmaD} in (\ref{vmed1a}), we deduce
%$$
%\left \vert\int_\Omega v^n_h \right\vert  \leq (1+k)^{-n}  \left\vert \int_\Omega v^{0}_h \right\vert + C_0\leq \left\vert \int_\Omega v^{0}_h \right\vert + C_0, \ \ \forall n\geq 0,
%$$
%which implies (\ref{weak02uv}). Finally,  from (\ref{lawenerfydisceuv}), summing for $m$ from $n_0+1$ to $n+n_0$, using (\ref{Ah1}), (\ref{weak01uv}), (\ref{weak02uv}) and the Poincar\'e inequality for the zero-mean value function $u^m_h -m_0$, where $ \displaystyle m_0= \frac{1}{\vert\Omega\vert}\int_\Omega u_0 =  \frac{1}{\vert\Omega\vert}\int_\Omega u^m_h$, we have $$\displaystyle k \underset{m=n_0+1}{\overset{n_0+n}{\sum}}\Vert  (u^m_h - m_0 ,  v^m_h)\Vert_{H^1\times W^{1,6}}^2\leq C_0 + C_1(nk),$$ 
%and thus, we deduce (\ref{weak01auv}). 
\end{proof}
Starting from the previous stability estimates,  the convergence towards weak solutions of (\ref{modelf00}) can be proved. Concretely, by introducing the functions:
\begin{itemize}
\item $(\widetilde{u}_{h,k},\widetilde{v}_{h,k})$ are continuous functions on $[0,+\infty)$, linear on each interval $(t_n,t_{n+1})$ and equal to $(u^n_h,{v}^n_h)$ at $t=t_n$, $n\geq 0$;
\item $({u}_{h,k},{v}_{h,k})$ are the piecewise constant functions taking values $(u^{n}_h,{v}^n_h)$ on $(t_{n-1},t_n]$, $n\geq 1$,
\end{itemize}
the following result holds: 
\begin{tma} {\bf (Convergence)}
There exist a subsequence $(k',h')$ of $(k,h)$, with $k',h'\downarrow 0$, and a weak-strong solution $(u,v)$ of (\ref{modelf00}) in $(0,+\infty)$, such that $(\widetilde{ u}_{h',k'}-m_0,\widetilde{v}_{h',k'}-m_0^2)$ and $(u_{h',k'}-m_0,v_{h',k'}-m_0^2)$  converge to $(u-m_0,v-m_0^2)$ weakly-$\star$ in $L^\infty(0,+\infty;L^2(\Omega)\times H^1(\Omega))$, weakly in $L^2(0,+\infty;H^1(\Omega)\times W^{1,6}(\Omega))$ and strongly in $L^2(0,T;L^2(\Omega)\times L^p(\Omega)) \cap C([0,T];H^1(\Omega)' \times L^q(\Omega))$, for any $T>0$, $1\leq p<+\infty$ and $1\leq q<6$. 
\end{tma}
\begin{obs}
Note that, since the positivity of $u^n_h$ cannot be assured, then the positivity of the limit function $u$ cannot be proven in the 3D case (see Remark \ref{OBSP}).
\end{obs}
\begin{proof}
Proceeding as in Theorem 4.11 of \cite{FMD} (whose proof follows the arguments of \cite{Tem}), one can prove that there exist a subsequence $(k',h')$ of $(k,h)$, with $k',h'\downarrow 0$, and $(u,v)$ satisfying (\ref{wf01}), (\ref{wf02}) and the initial conditions (\ref{modelf00})$_4$, such that $(\widetilde u_{h',k'}-m_0,\widetilde{v}_{h',k'}-m_0^2)$ and $(u_{h',k'}-m_0,v_{h',k'}-m_0^2)$  converge to $(u-m_0,v-m_0)$ weakly-* in $L^\infty(0,+\infty;L^2(\Omega)\times H^1(\Omega))$, weakly in $L^2(0,+\infty;H^1(\Omega)\times W^{1,6}(\Omega))$ and strongly in $L^2(0,T;L^2(\Omega)\times L^p(\Omega)) \cap C([0,T];H^1(\Omega)' \times L^q(\Omega))$, for any $T>0$, $1\leq p<+\infty$ and $1\leq q<6$. Moreover, it holds
\begin{eqnarray*}
&\displaystyle\frac{d}{dt}\left( \frac{1}{2}\Vert \widetilde{u}_{k',h'}(t) \Vert_0^2 + \frac{1}{4} \Vert\nabla \widetilde{v}_{k',h'}(t)\Vert_0^2\right)&\!\!\! + \frac{(t_n - t)}{2}\Vert (\delta_t u_n ,\delta_t \nabla {v}_n) \Vert_0^2 \nonumber\\
&&\hspace{-2 cm} + \Vert \nabla u_{k',h'}(t)\Vert_{0}^{2} +
\displaystyle\frac{1}{2}\Vert (A_h - I) {v}_{k',h'}(t)\Vert_{0}^{2}+
\displaystyle\frac{1}{2}\Vert \nabla {v}_{k',h'}(t)\Vert_{0}^{2}= 0.
\end{eqnarray*}
In order to obtain that $(u,v)$ satisfies the energy inequality (\ref{wsd}), it is necessary to prove that 
\begin{equation}\label{uny}
\underset{(k',h')\rightarrow (0,0)}{\lim \mbox{inf} } \int_{t_0}^{t_1}  \Vert(A_h - I) v_{k',h'}(t)\Vert_{0}^{2} \geq \int_{t_0}^{t_1}  \Vert\Delta v(t)\Vert_{0}^{2}.
\end{equation}
Taking into account that $\{(A_h - I) v_{k',h'}\}$ is bounded in $L^2(0,T;L^2(\Omega))$, one has that there exists $w\in L^2(0,T;L^2(\Omega)$ such that for some subsequence of $(k',h')$, still denoted by $(k',h')$, 
\begin{equation}\label{ca3}
(A_h - I) v_{k',h'} \rightarrow w \ \ \mbox { weakly in } \ L^2(0,T;L^2(\Omega).
\end{equation}
Since $u^2\in L^2(0,T;L^{3/2}(\Omega))\hookrightarrow L^2(0,T;H^{1}(\Omega)')$, one has
\begin{equation}\label{ca4}
\partial_t v - \Delta v +v=u^2 \ \ \mbox{ in } L^2(H^1)',
\end{equation}
and, on the other hand, using (\ref{ca3}), one can deduce 
\begin{equation}\label{ca5}
\partial_t v +w +v=u^2 \ \ \mbox{ in } L^2(H^1)'.
\end{equation}
Thus, from (\ref{ca4})-(\ref{ca5}), one can deduce that $w=-\Delta v$ in $\mathcal{D}'(\Omega)$, which implies $-\Delta v \in L^2(0,T;L^2(\Omega))$ because of $w\in L^2(0,T;L^2(\Omega)$. Therefore, $(u,v)$ satisfies the regularity (\ref{wsa}) and taking into account (\ref{ca3}),  (\ref{uny}) is concluded. Finally, using (\ref{uny}) and arguing as in the last part of the proof of Theorem 4.11 of \cite{FMD}, it can be obtained that $(u,v)$ satisfies the energy inequality (\ref{wsd}), and therefore, $(u,v)$ is a weak-strong solution of (\ref{modelf00}).
\end{proof}

\subsection{Large-time behavior of the scheme UV}
In this subsection, exponential bounds for any solution  $(u_h^n,v_h^n)$ of the scheme \textbf{UV} in weak-strong norms are proved. 
\begin{tma}\label{CoUV}
Let $(u_h^n,v^n_h)$ be a solution of the scheme \textbf{UV} associated to an initial data $(u^{0}_h,v^0_h)$, with $\displaystyle\frac{1}{\vert\Omega\vert}\int_\Omega u^0_h = \displaystyle\frac{1}{\vert\Omega\vert}\int_\Omega u_0 = m_0$. Then,
\begin{equation}\label{LTuv1}
 \displaystyle\Vert (\hat{u}^n_h,\nabla v^n_h)\Vert_{0}^{2}\leq C_0(1+2K_p k)^{-n}, \ \ \forall n\geq 0,
\end{equation}
\begin{equation}\label{LTuv2}
 \displaystyle\Vert \hat{v}^n_h \Vert_{0}^{2}\leq 
 \left\{\begin{array}{l}
 C (1+ k)^{-n} \ \mbox{ if } 2K_p>1,\\
 C(kn) (1+ k)^{-n}  \ \mbox{ if } 2K_p=1,\\
C (1+ 2K_p k)^{-n}  \ \mbox{ if } 2K_p<1,
 \end{array}\right.
\end{equation}
%\begin{equation}\label{LTuv1-a}
%k\underset{m>n}{{\sum}} \Big(\Vert \hat{u}^m_h\Vert_1^2 +\displaystyle\frac{1}{2}\Vert (A_h-I) v^m_h\Vert_{0}^{2} +
%\displaystyle\frac{1}{2}\Vert \nabla v^m_h\Vert_{0}^{2}\Big)\leq C_0e^{- \frac{K_p}{1+K_pk}kn}, \ \ \forall n\geq 0,
%\end{equation}
where the constant $K_p>0$ was defined in Subsection \ref{subcc}.
\end{tma}
\begin{proof}
Taking $\bar{u}_h=\hat{u}^n_h$ in (\ref{modelf02uv})$_1$, $\bar{v}_h= \displaystyle\frac{1}{2}(A_h -I) v^n_h$
in (\ref{refvh}) and using (\ref{consuuv}) and (\ref{discVh}), one obtains
\begin{eqnarray}\label{Com1}
&\delta_t \Big( \displaystyle
\frac{1}{2}\Vert \hat{u}^n_h\Vert_{0}^2 + \frac{1}{4}\Vert \nabla v^n_h\Vert_{0}^{2} \Big)&\!\!\!\!\!+ 
\frac{k}{2} 
\Vert \delta_t \hat{u}^n_h\Vert_{0}^2 +
\frac{k}{4} \Vert \delta_t \nabla v^n_h\Vert_{0}^2 \nonumber\\
&&\!\!\!
 + \Vert \hat{u}^n_h\Vert_{1}^{2} +
\displaystyle\frac{1}{2}\Vert (A_h-I) v^n_h\Vert_{0}^{2} +
\displaystyle\frac{1}{2}\Vert \nabla v^n_h\Vert_{0}^{2}=0.
\end{eqnarray}
To get (\ref{Com1}), the fact that $(u^n_h)^2 \in V_h$ is essential (which comes from the choice $(\mathbb{P}_m,\mathbb{P}_{2m})$ approximation for $(U_h,V_h)$) in order to cancel the terms $(u^n_h\nabla v^n_h,\nabla \hat{u}^n_h)$ and $-\frac{1}{2}((u^n_h)^2,(A_h -I) v^n_h)$. Then, from (\ref{Com1}) one arrives at
\begin{equation*}
(1+2K_p k) \Big( \displaystyle\Vert \hat{u}^n_h\Vert_{0}^2 + \frac{1}{2} \Vert \nabla v^n_h\Vert_{0}^{2} \Big) - \Big( \Vert \hat{u}^{n-1}_h\Vert_{0}^2 + \frac{1}{2} \Vert \nabla v^{n-1}_h\Vert_{0}^{2}  \Big) \leq 0,
\end{equation*}
from which, multiplying by $(1+2K_p k)^{n-1}$ and summing, one has for all $n\geq 0$,
\begin{equation}\label{e33}
 \displaystyle\Vert \hat{u}^n_h\Vert_{0}^2 + \frac{1}{2} \Vert \nabla v^n_h\Vert_{0}^{2}\leq (1+2K_p k)^{-n} \Big( \Vert \hat{u}^0_h\Vert_{0}^2 + \frac{1}{2} \Vert \nabla v^0_h\Vert_{0}^{2} \Big)  
\end{equation}
and (\ref{LTuv1}) is obtained. Moreover, taking $\bar{v}_h= \hat{v}^n_h$ in (\ref{vm0}), one has
$$
\frac{1}{2} \delta_t \Vert \hat{v}^n_h\Vert_0^2 + \Vert  \hat{v}^n_h\Vert_1^2 =\int_\Omega (\hat{u}^n_h + 2m_0)\hat{u}^n_h \hat{v}^n_h,
$$
which, using the H\"older and Young inequalities, implies that
\begin{equation}\label{Com4}
(1+ k) \Vert \hat{v}^n_h\Vert_0^2 - \Vert  \hat{v}^{n-1}_h\Vert_0^2\leq  kC \Vert \hat{u}^n_h +2m_0\Vert_{L^3}^2\Vert \hat{u}^n_h \Vert_0^2.
\end{equation}
Then, multiplying (\ref{Com4}) by $(1+ k)^{n-1}$, summing and using (\ref{LTuv1}), one obtains 
\begin{equation}\label{Newv1}
(1+ k)^{n}\Vert \hat{v}^n_h\Vert_0^2\leq  \Vert \hat{v}^0_h\Vert_0^2 +  \frac{C}{1+2K_pk}  k \underset{m=1}{\overset{n}{\sum}} \left(\frac{1 +k}{1+2K_pk}\right)^{m-1}(1 +\Vert \hat{u}^m_h\Vert^2_1).
\end{equation}
Then, in order to obtain (\ref{LTuv2}) we split the argument in three cases:
\begin{enumerate}
	\item Case 1: If $2K_p= 1$, using (\ref{weak01uv}) in (\ref{Newv1}) one has that for any $t_n=nk>1$,
\begin{equation}\label{Newv2}
\Vert \hat{v}^n_h\Vert_0^2\leq  (1+ k)^{-n} (C+ C(kn))\leq  C(kn) (1+ k)^{-n}.
\end{equation}

	\item Case 2: If $2K_p>1$, using (\ref{weak01uv}) in (\ref{Newv1}) one obtains
	\begin{equation}\label{Newv3} 
	\Vert \hat{v}^n_h\Vert_0^2\leq  (1+ k)^{-n}\!\left(C_0 + \frac{C}{2K_p-1}\!\left[1 \!-\!\left(\frac{1+k}{1+2K_pk}\right)^n\right] +\frac{C}{1+2K_pk} \right)\leq C(1+ k)^{-n}.
	\end{equation}
	
		\item Case 3: If $2K_p<1$, one rewrites (\ref{Newv1}) as
	\begin{equation*}
		(1+ 2K_p k)^{n}\Vert \hat{v}^n_h\Vert_0^2\leq  \left(\frac{1+ 2K_p k}{1+k}\right)^{n}\Vert \hat{v}^0_h\Vert_0^2 +  \frac{C}{1+2K_pk}  k \underset{m=1}{\overset{n}{\sum}} \left(\frac{1 +2K_p k}{1+k}\right)^{n-m+1}\!\!(1 +\Vert \hat{u}^m_h\Vert^2_1),
	\end{equation*}
		and proceeding as in (\ref{Newv3}), taking into account that $\frac{1+ 2K_p k}{1+k}<1$, one arrives at
	\begin{equation}\label{Newv4} 
	\Vert \hat{v}^n_h\Vert_0^2\leq C(1+2K_p k)^{-n}.
	\end{equation}
\end{enumerate}
Therefore, from (\ref{Newv2})-(\ref{Newv4}), (\ref{LTuv2}) is deduced.
\end{proof}

\begin{cor}
Under conditions of Theorem \ref{CoUV}, the following estimates hold
\begin{equation*}
\displaystyle\Vert (\hat{u}^n_h,\nabla v^n_h)\Vert_{0}^{2}\leq C_0e^{- \frac{2K_p}{1+2K_pk}kn}, \ \ \forall n\geq 0,
\end{equation*}
\begin{equation*}\label{LTuv2-cor}
\displaystyle\Vert \hat{v}^n_h \Vert_{0}^{2}\leq
 \left\{\begin{array}{l}
C e^{- \frac{1}{1+k}kn} \ \ \mbox{ if } 2K_p>1,\\
C(kn) e^{- \frac{1}{1+k}kn}  \ \mbox{ if } 2K_p=1,\\
C e^{- \frac{2K_p}{1+2K_p k}kn}   \ \mbox{ if } 2K_p<1.
\end{array}\right.
\end{equation*}
\end{cor}
\begin{proof}
Using the inequality $1-x\leq e^{-x}$ for all $x\geq 0$, from (\ref{LTuv1}) one has
\begin{equation}\label{cnewu}
 \displaystyle\Vert (\hat{u}^n_h,\nabla v^n_h)\Vert_{0}^{2}\leq C_0(1+2K_p k)^{-n}\leq C_0\Big(1 - \frac{2K_p}{1+2K_p k}k\Big)^{n}\leq C_0e^{- \frac{2K_p}{1+2K_p k}kn}.
 \end{equation}
Analogously, (\ref{LTuv2}) can be deduced.

\end{proof}

\section{Scheme US$_\varepsilon$}\label{schemeUSe}
Up to our knowledge, there is not previous works studying FE schemes for model (\ref{modelf00}), with positive or approximately positive discrete solutions. In fact, for the scheme \textbf{UV} analyzed in this paper or the scheme \textbf{US} studied in \cite{FMD2}, it is not clear how to prove any of these properties. For this reason, in this section we propose an unconditionally energy-stable scheme with the property of ``approximated positivity''; this scheme is constructed as a modification of the scheme \textbf{US} (\cite{FMD2}), by introducing the auxiliary variable ${\boldsymbol{\sigma}}=\nabla v$ and applying a regularization procedure.

\

We consider a fully discrete approximation using FE in space and backward Euler in time for a reformulated problem in $(u,{\boldsymbol{\sigma}})$-variables. Moreover, in this case we will assume the following hypothesis on the space discretization:
\begin{enumerate}
	\item[({\bf H})]{The triangulation is structured in the sense that all simplices have a right angle.}
\end{enumerate}
We choose the following continuous FE spaces for $u$, ${\boldsymbol{\sigma}}$ and $v$:
$$ (U_h, {\boldsymbol\Sigma}_h, V_h) \subset H^1 \times \mathbf{H}^1_{\sigma}\times W^{1,6} \quad \hbox{generated by $\mathbb{P}_1$-continuous FE.}
$$
\begin{obs}
	The right angled requirement and the choice of $\mathbb{P}_1$-continuous FE for $U_h$ are nece\-ssa\-ry in order to obtain the relation (\ref{PL1}) below, which is essential in order to obtain the approximated positivity  (see Theorem \ref{AAPP} below).
\end{obs}
We consider the Lagrange interpolation operator $\Pi^h: C(\overline{\Omega})\rightarrow U_h$, and we introduce the discrete semi-inner product on $C(\overline{\Omega})$ (which is an inner product in $U_h$) and its induced discrete seminorm (norm in $U_h$):
\begin{equation*}
(u_1,u_2)^h:=\int_\Omega \Pi^h (u_1 u_2), \   \vert u \vert_h=\sqrt{(u,u)^h}.
\end{equation*}
\begin{obs}\label{eqh2}
	In $U_h$, the norms $\vert \cdot\vert_h$ and $\Vert \cdot\Vert_0$ are equivalents uniformly with respect to $h$ (see \cite{PB}).
\end{obs}
We consider also the  $L^2$-projection $Q^h:L^2(\Omega)\rightarrow U_h$ given by
\begin{equation*}
(Q^h u,\bar{u})^h=(u,\bar{u}), \ \ \forall \bar{u}\in U_h,
\end{equation*}
the standard $L^2$-projection $\widetilde{Q}^h:L^2(\Omega)\rightarrow {\boldsymbol\Sigma}_h$. Moreover, following the ideas of Barrett and Blowey \cite{BB}, we consider the truncated function $\lambda_\varepsilon:\mathbb{R}\rightarrow [\varepsilon,\varepsilon^{-1}]$ (with $\varepsilon\in (0,1)$) given by
\begin{equation*}
\lambda_\varepsilon(s)\ := \ 
\left\{\begin{array}{lcl}
\varepsilon & \mbox{ if } & s\leq \varepsilon,\\
s & \mbox{ if } & \varepsilon\leq s\leq \varepsilon^{-1},\\
\varepsilon^{-1} & \mbox{ if } & s\geq \varepsilon^{-1}.
\end{array}\right. 
\end{equation*}
If we define 
\begin{equation} \label{F2pE}
F''_\varepsilon(s):= \frac{1}{\lambda_\varepsilon(s)},
\end{equation}
then, we can integrate twice in (\ref{F2pE}), imposing the conditions $F'_\varepsilon(1)=F_\varepsilon(1)=0$, and we obtain a convex function $F_\varepsilon: \mathbb{R}\rightarrow [0,+\infty)$, such that $F_\varepsilon \in  C^{2}(\mathbb{R})$. Even more, for $\varepsilon\in (0,e^{-2})$, it holds (see \cite{BB})
\begin{equation}\label{PNa}
F_\varepsilon(s)\geq \frac{\varepsilon}{2}s^2 - 2\ \ \forall s\geq 0 \ \  \mbox{  and } \ \ F_\varepsilon(s)\geq \frac{s^2}{2\varepsilon} \ \ \forall s\leq 0.
\end{equation}
Then, for each $\varepsilon\in (0,1)$ we consider the construction of the operator $\Lambda_\varepsilon: U_h\rightarrow L^\infty(\Omega)^{d\times d}$ given in  \cite{BB}, satisfying that $\Lambda_\varepsilon u^h$ is a piecewise constant matrix  for all $u^h\in U_h$, such that the following relation holds  
\begin{equation}\label{PL1}
(\Lambda_\varepsilon u^h) \nabla \Pi^h (F'_\varepsilon(u^h))=\nabla u^h \ \ \mbox{ in } \Omega.
\end{equation}
Basically, $\Lambda_\varepsilon u^h$ is a constant by elements symmetric and positive definite matrix such that (\ref{PL1}) holds by elements. We highlight that (\ref{PL1}) is satisfied due to the right angled constraint requirement ({\bf H}) and the choice of $\mathbb{P}_1$-continuous FE for $U_h$. We recall the result below concerning to $\Lambda_\varepsilon(\cdot)$ (see \cite[Lemma 2.1]{BB}).
\begin{lem}\label{lemconv}
	Let $\Vert \cdot \Vert$ denote the spectral norm on $\mathbb{R}^{d\times d}$. Then for any given $\varepsilon\in (0,1)$ the function $\Lambda_\varepsilon:U_h\rightarrow [L^\infty(\Omega)]^{d\times d}$ is continuous and satisfies
	\begin{equation}\label{D}
	\varepsilon \xi^T \xi \leq \xi^T \Lambda_\varepsilon(u^h) \xi \leq \varepsilon^{-1} \xi^T \xi, \ \ \forall \xi \in \mathbb{R}^d, \ \forall u^h\in U_h.
	\end{equation}
\end{lem}
 Then, the following first order in time, nonlinear and coupled scheme is considered: 
\begin{itemize}
	\item{\underline{\emph{Scheme \textbf{US}$_\varepsilon$}:}\\
		{\bf Initialization}: Let $(u^{0}_h,{\boldsymbol \sigma}_h^{0},{v}_h^{0})=(Q^h u_0,\widetilde{Q}^h (\nabla v_0),Q^h {v}_{0})\in  U_h\times {\boldsymbol\Sigma}_h\times V_h$. Then, $\displaystyle\frac{1}{\vert\Omega\vert}\int_\Omega u^0_h = \displaystyle\frac{1}{\vert\Omega\vert}\int_\Omega u_0 = m_0$, $u^0_h\geq 0$ and $v^0_h\geq 0$.\\
		{\bf Time step} n: Given $(u^{n-1}_\varepsilon,{\boldsymbol \sigma}^{n-1}_\varepsilon)\in  U_h\times {\boldsymbol\Sigma}_h$, compute $(u^{n}_\varepsilon,{\boldsymbol \sigma}^{n}_\varepsilon)\in  U_h \times {\boldsymbol\Sigma}_h$ solving
		\begin{equation}
		\left\{
		\begin{array}
		[c]{lll}%
		(\delta_t u^n_\varepsilon,\bar{u})^h + (\nabla u^n_\varepsilon,\nabla \bar{u}) + (\Lambda_\varepsilon (u^{n}_\varepsilon){\boldsymbol\sigma}^n_\varepsilon,\nabla \bar{u})= 0, \ \ \forall \bar{u}\in U_h,\\
			(\delta_t {\boldsymbol \sigma}^n_\varepsilon,\bar{\boldsymbol \sigma}) + ( B_h {\boldsymbol \sigma}^n_\varepsilon,\bar{\boldsymbol \sigma}) =
		2(\Lambda_\varepsilon (u^{n}_\varepsilon) \nabla  u^n_\varepsilon,\bar{\boldsymbol \sigma}),\ \ \forall
		\bar{\boldsymbol \sigma}\in \Sigma_h,
		\end{array}
		\right.  \label{modelf02a}
		\end{equation}}
	where the linear operator $B_h:{\boldsymbol\Sigma}_h \rightarrow {\boldsymbol\Sigma}_h$ is defined as
	$$( B_h {\boldsymbol \sigma}^n_\varepsilon,\bar{\boldsymbol \sigma}) =(\nabla \cdot {\boldsymbol \sigma}^n_\varepsilon,\nabla \cdot \bar{\boldsymbol\sigma}) + (\mbox{rot }{\boldsymbol \sigma}^n_\varepsilon,\mbox{rot }\bar{\boldsymbol\sigma}) + ({\boldsymbol \sigma}^n_\varepsilon,\bar{\boldsymbol\sigma}), \  \  \forall \bar{\boldsymbol\sigma} \in {\boldsymbol\Sigma}_h.$$
\end{itemize}
Once the scheme \textbf{US} is solved, given $v^{n-1}_\varepsilon\in V_h$, we can recover $v^n_\varepsilon=v^n_\varepsilon((u^n_\varepsilon)^2) \in V_h$ solving: 
\begin{equation}\label{edovfUSe}
(\delta_t v^n_\varepsilon,\bar{v})^h  +(\nabla v^n_\varepsilon,\nabla \bar{v}) + (v^n_\varepsilon,\bar{v})^h =((u^n_\varepsilon)^2,\bar{v}), \ \ \forall
\bar{v}\in V_h.
\end{equation}

\subsection{Mass-conservation, well-posedness and energy-stability}\label{ELus}
In this subsection, we are going to present some properties of the scheme \textbf{US}$_\varepsilon$, whose proofs follow the ideas of the scheme \textbf{US} studied in \cite{FMD2}. We highlight that these properties can be obtained independently of the choice of $\Lambda_\varepsilon(u^n_\varepsilon)$ approximating $u^n_\varepsilon$.

\
 
Since $\bar{u}=1\in U_h$ and $\bar{v}=1 \in V_h$, then the scheme \textbf{US}$_\varepsilon$ is conservative in $u^n_\varepsilon$, that is,
\begin{equation}\label{conu1N}
(u_\varepsilon^n,1)=(u^n_\varepsilon,1)^h= (u^{n-1}_\varepsilon,1)^h=\cdot\cdot\cdot= (u^{0}_h,1)^h=(u_h^0,1)=(Q^hu_0,1)=(u_0,1)=m_0\vert \Omega\vert,
\end{equation}
and also has the  behavior for $\int_\Omega v^n_\varepsilon$ given in \eqref{compv1uv} (with $u^n_\varepsilon$ and $v^n_\varepsilon$ instead of $u^n_h$ and $v^n_h$ respectively).

\

In the following results we stablish the well-posedness of problems (\ref{modelf02a}) and  (\ref{edovfUSe}).
\begin{tma} {\bf (Unconditional solvability and conditional uniqueness  of (\ref{modelf02a}))} 
	There exists at least one solution $(u^n_\varepsilon,{\boldsymbol\sigma}^n_\varepsilon)$  of scheme \textbf{US}$_\varepsilon$. Moreover, 	if $k\, f(h,\varepsilon)<1$ (where $f(h,\varepsilon)\uparrow +\infty$ when $h\downarrow 0$ or $\varepsilon\downarrow 0$), then the solution $(u^n_\varepsilon,{\boldsymbol \sigma}_\varepsilon^n)$ of the scheme \textbf{US}$_\varepsilon$ is unique.
\end{tma}
\begin{proof}
The proof of solvability follows as in Theorem 3.13 (see also Theorem 4.6) of \cite{GRR1}, and the uniqueness
follows as in Lemma 3.14 (see also Lemma 4.7) of \cite{GRR1}.
\end{proof}

\begin{lem}{\bf (Well-posedness of (\ref{edovfUSe}))}
	Given $u^n_\varepsilon\in U_h$ and $v^{n-1}_\varepsilon\in V_h$, there exists a unique $v^n_\varepsilon \in V_h$ solution of (\ref{edovfUSe}).
\end{lem}
\begin{proof}
	The proof follows from Lax-Milgram theorem.
\end{proof}

\begin{defi}\label{enesf00}
	A numerical scheme with solution $(u^n_\varepsilon,{\boldsymbol{\sigma}}^n_\varepsilon)$ is called energy-stable  if the energy
	\begin{equation}\label{nueva-2}
	\widetilde{\mathcal{E}}(u,{\boldsymbol \sigma})=\frac{1}{2}\Vert u\Vert_{0}^2 + \frac{1}{4}\Vert {\boldsymbol
		\sigma}\Vert_{0}^{2}
	\end{equation}
	is time decreasing, that is, 
	\begin{equation}\label{stabf0200}
	\widetilde{\mathcal{E}}(u^n_\varepsilon,{\boldsymbol \sigma}^n_\varepsilon)\leq \widetilde{\mathcal{E}}(u^{n-1}_\varepsilon,{\boldsymbol \sigma}^{n-1}_\varepsilon), \ \ \forall n\geq 1.
	\end{equation}
\end{defi}

\begin{tma} {\bf (Unconditional stability)} \label{estinc1us}
	The scheme \textbf{US}$_\varepsilon$ is unconditionally energy stable with respect to the modified energy $\widetilde{\mathcal{E}}(u,{\boldsymbol\sigma})$ given in (\ref{nueva-2}). In fact, if $(u^n_\varepsilon,{\boldsymbol\sigma}^n_\varepsilon)$ is a solution of \textbf{US}$_\varepsilon$, then the following discrete energy law holds
	\begin{equation}\label{delus}
	\delta_t \widetilde{\mathcal{E}}(\hat{u}^n_\varepsilon,{\boldsymbol\sigma}^n_\varepsilon) +\frac{k}{2}\Vert \delta_t \hat{u}^n_\varepsilon\Vert_0^2 + \frac{k}{4} \Vert \delta_t  {\boldsymbol\sigma}^n_\varepsilon\Vert_0^2 + \Vert  \hat{u}^n_\varepsilon \Vert_1^2  +\frac{1}{2}\Vert  {\boldsymbol\sigma}^n_\varepsilon\Vert_1^2 \leq 0.
	\end{equation}
\end{tma}
\begin{proof}
	Testing (\ref{modelf02a})$_1$ by $\bar{u}= \hat{u}^n_\varepsilon$, (\ref{modelf02a})$_2$ by $\bar{\boldsymbol\sigma}=\frac{1}{2}{\boldsymbol\sigma}^n_\varepsilon$ and adding, the terms
	$(\Lambda_\varepsilon (u^{n}_\varepsilon) \nabla  \hat{u}_\varepsilon^n,{\boldsymbol \sigma}^n_\varepsilon)$ cancel, and taking into account Remark \ref{eqh2},  (\ref{delus}) is obtained.
\end{proof}
From (\ref{delus}), multiplying by $k$ and summing, one can deduce the following global energy law:
\begin{cor} \label{welemUS} {\bf(Global energy law) }
	Assume that $(u_0,v_0)\in L^2(\Omega)\times H^1(\Omega)$. Let $(u^n_\varepsilon,{\boldsymbol\sigma}^n_\varepsilon)$ be any solution of scheme \textbf{US}$_\varepsilon$. Then, the following estimate holds
	\begin{equation*}\label{gest}
	\Vert (\hat{u}^n_\varepsilon, {\boldsymbol
		\sigma}^n_\varepsilon)\Vert_{0}^{2}
	%+ k^2 \underset{m=1}{\overset{n}{\sum}} \Vert (\delta_t u^m_h,\delta_t {\boldsymbol \sigma}^m_h) \Vert_0^2
	+ k \underset{m=1}{\overset{n}{\sum}}\Vert (\hat{u}^m_\varepsilon, {\boldsymbol \sigma}^m_\varepsilon)\Vert_{1}^2\leq C_0, \ \ \ \forall n\geq 1.
	\end{equation*}
\end{cor}

\subsection{Large-time behavior of scheme \textbf{US}$_\varepsilon$}

\begin{tma}\label{LTBusEPS}
	Let $(u^n_\varepsilon,{\boldsymbol \sigma}^n_\varepsilon)$ be any solution of the scheme \textbf{US}$_\varepsilon$. Then, the following estimate holds 
	\begin{equation}\label{LTusEPS}
	\displaystyle\Vert (\hat{u}^n_\varepsilon,{ \boldsymbol \sigma}^n_\varepsilon)\Vert_{0}^{2}\leq C_0e^{- \frac{2K_p}{1+2K_p k}kn}, \ \ \forall n\geq 0,
	\end{equation}
	where the constant $K_p>0$ was defined in Subsection  \ref{subcc}.
\end{tma}
\begin{proof}
	Taking $\bar{u}_h=\hat{u}^n_\varepsilon$ in (\ref{modelf02a})$_1$, $\bar{\boldsymbol \sigma}_h= \displaystyle\frac{1}{2}{\boldsymbol \sigma}^n_\varepsilon$
	in (\ref{modelf02a})$_2$ and using (\ref{conu1N}) as well as Remark \ref{eqh2}, one obtains
	\begin{equation*}
	\delta_t \Big( \displaystyle
	\frac{1}{2}\Vert \hat{u}^n_\varepsilon\Vert_{0}^2 + \frac{1}{4}\Vert{\boldsymbol \sigma}^n_\varepsilon\Vert_{0}^{2} \Big)+ 
	\frac{k}{2} 
	\Vert \delta_t \hat{u}^n_\varepsilon\Vert_{0}^2 +
	\frac{k}{4} \Vert \delta_t {\boldsymbol \sigma}^n_\varepsilon\Vert_{0}^2 
	+ \Vert \hat{u}^n_\varepsilon\Vert_{1}^{2} +
	\displaystyle\frac{1}{2}\Vert{ \boldsymbol \sigma}^n_\varepsilon\Vert_{1}^{2}=0.
	\end{equation*}
	Then,  proceeding as in (\ref{e33}) and (\ref{cnewu}), one arrives at (\ref{LTusEPS}).
\end{proof}
\begin{cor}
	Let $v^n_\varepsilon=v^n_\varepsilon((u^n_\varepsilon)^2)$ be a solution of (\ref{edovfUSe}). Then, it holds 
	\begin{equation*}
	\displaystyle\Vert \hat{v}^n_\varepsilon \Vert_{0}^{2}\leq
	\left\{\begin{array}{l}
	C e^{- \frac{1}{1+k}kn} \ \ \mbox{ if } 2K_p>1,\\
	C(kn) e^{- \frac{1}{1+k}kn}  \ \mbox{ if } 2K_p=1,\\
	C e^{- \frac{2K_p}{1+2K_p k}kn}   \ \mbox{ if } 2K_p<1,
	\end{array}\right.
	\end{equation*}
	where the constant $K_p>0$ was defined in Subsection  \ref{subcc}.
\end{cor}
\begin{proof}
	The proof follows as in Theorem \ref{CoUV}; using, in this case, Remark \ref{eqh2}.
\end{proof}

\subsection{Positivity of $v^n_\varepsilon$ and approximated positivity of $u^n_\varepsilon$}
First, the positivity of the discrete chemical signal will be proved. For this, it will be essential that the interior angles of the triangles or tetrahedra be less than or equal to $\pi/2$. Since we impose  the right angled constraint (\textbf{H}), then this property holds. 
\begin{lem}{\bf (Positivity of $v^n_\varepsilon$)}
	Given $u^n_\varepsilon\in U_h$ and $v^{n-1}_\varepsilon\in V_h$, the unique $v^n_\varepsilon \in V_h$ solution of (\ref{edovfUSe}) satisfies $v^n_\varepsilon\geq 0$.
\end{lem}
\begin{proof}
We define $v^n_{\varepsilon-}:=\min\{v^n_{\varepsilon},0\}$ and $v^n_{\varepsilon+}:=\max\{v^n_{\varepsilon},0\}$. Then, testing (\ref{edovfUSe}) by $\bar{v}=\Pi^h (v^n_{\varepsilon-}) \in V_h$, and taking into account that $(\nabla \Pi^h (v^n_{\varepsilon+}), \nabla \Pi^h (v^n_{\varepsilon-}))\geq 0$ (owing to the interior angles of the triangles or tetrahedra are less than or equal to $\pi/2$), and using that $(\Pi^h (v))^2\leq \Pi^h(v^2)$ for all $v\in C(\overline{\Omega})$, one has
	\begin{eqnarray*}
		&\displaystyle\Big(\frac{1}{k} + 1 \Big) \Vert \Pi^h (v^n_{\varepsilon-})\Vert_0^2&\!\!\! + \Vert \nabla \Pi^h (v^n_{\varepsilon-})\Vert_0^2 \leq 0,
	\end{eqnarray*}
	and the proof is concluded.
\end{proof}

Notice that the above properties were proved independently of the choice of $\Lambda_\varepsilon(u^n_\varepsilon)$ approximating $u^n_\varepsilon$.  Now, in order to obtain aproximated positivity for the discrete cell density $u^n_\varepsilon$, we need to consider $\Lambda_\varepsilon(u^n_\varepsilon)$ satisfying (\ref{PL1}) and (\ref{D}). The main idea in the proof is to get the following bound
\begin{equation}\label{pppu}
(F_\varepsilon(u^n_\varepsilon),1)^h\leq C
\end{equation}
which, following the ideas of Corallary 3.9 and Remark 3.12 of \cite{GRR1}, implies the estimate (\ref{UPosi}) below, from which one can deduce that $u^n_{\varepsilon -}\rightarrow 0$ as $\varepsilon\rightarrow 0$ in the $L^2(\Omega)$-norm.
\begin{tma}{\bf(Approximated positivity of $u^n_\varepsilon$) }\label{AAPP}
Let $(u^n_\varepsilon,{\boldsymbol\sigma}^n_\varepsilon)$ any solution of the scheme \textbf{US}$_\varepsilon$. If $\varepsilon\in (0,e^{-2})$,  the following estimate holds
\begin{equation}\label{UPosi}
\max_{n\geq 0} \Vert \Pi^h (u^n_{\varepsilon-})\Vert_0^2 \leq C_0\varepsilon,
\end{equation}
where the constant $C_0$ depends on the data $(\Omega,u_0, v_0)$, but is independent of $k,h,n$ and $\varepsilon$.
\end{tma}
\begin{proof}
Testing (\ref{modelf02a})$_1$ by $\bar{u}= \Pi^h (F'_\varepsilon (u^n_\varepsilon))$ and taking into account that $\Lambda_\varepsilon(u^n_\varepsilon)$ is symmetric as well as (\ref{PL1}) (which implies that $\nabla \Pi^h (F'_\varepsilon (u^n_\varepsilon))=\Lambda_\varepsilon^{-1} (u^{n}_\varepsilon) \nabla u^n_\varepsilon$), one obtains
\begin{eqnarray}\label{I01a}
(\delta_t u^n_\varepsilon,\Pi^h (F'_\varepsilon (u^n_\varepsilon)))^h+ \int_\Omega (\nabla u^n_\varepsilon)^T\!\cdot\!\Lambda_\varepsilon^{-1} (u^{n}_\varepsilon)\!\cdot\! \nabla u^n_\varepsilon d\x  = -\int_\Omega {\boldsymbol{\sigma}^n_\varepsilon}\cdot \nabla u^n_\varepsilon d\x.
\end{eqnarray}
By using the Taylor formula and taking into account that $\Pi^h$ is linear and $F''_\varepsilon(s)\geq \varepsilon$ for all $s\in \mathbb{R}$, one has (following \cite[Theorem 3.8]{GRR1})
\begin{equation*}
(\delta_t u^n_\varepsilon,\Pi^h (F'_\varepsilon (u^n_\varepsilon)))^h \geq \delta_t (F_\varepsilon(u^n_\varepsilon),1)^h  + \varepsilon\frac{k}{2}\vert \delta_t u^n_\varepsilon \vert_h^2,
\end{equation*}
which, together with (\ref{D}), (\ref{I01a}) and Remark \ref{eqh2}, imply that 
\begin{equation}\label{deluvNN}
\delta_t (F_\varepsilon(u^n_\varepsilon),1)^h+ \varepsilon\frac{k}{2}\Vert \delta_t u^n_\varepsilon \Vert_0^2  + \varepsilon\Vert \nabla u^n_\varepsilon\Vert_0^2  \leq \frac{1}{2} \Vert \nabla u^n_\varepsilon\Vert_0^2 + \frac{1}{2} \Vert{\boldsymbol{\sigma}}^n_\varepsilon\Vert_0^2.
\end{equation}
Then, multiplying (\ref{deluvNN}) by $k$, adding for $n=1,\cdot\cdot\cdot,m$ and using Corollary \ref{welemUS}, one arrives at
\begin{equation*}
(F_\varepsilon(u^m_\varepsilon),1)^h \leq  (F_\varepsilon(u^0_h),1)^h+k \underset{n=1}{\overset{m}{\sum}}\left( \frac{1}{2} \Vert \nabla u^n_\varepsilon\Vert_0^2 + \frac{1}{2} \Vert{\boldsymbol{\sigma}}^n_\varepsilon\Vert_0^2\right) \leq C_0, 
\end{equation*}
where $C_0>0$ is a constant depending on the data $(\Omega, u_0, v_0)$, but independent of $k,h,n$ and $\varepsilon$. Thus, (\ref{pppu}) is obtained. Therefore,  if $\varepsilon\in (0,e^{-2})$, from (\ref{PNa})$_2$ and following the proof of Corollary 3.9 and Remark 3.12 of \cite{GRR1}, (\ref{UPosi}) is deduced.
\end{proof}

\section{Numerical Simulations}\label{NSi}
In this section we will compare the results of several numerical si\-mu\-lations that we have carried out using the schemes studied in the paper. 
We are considering 
 $(\mathbb{P}_1,\mathbb{P}_2)$-continuous approximation for $(u^n_h,  v_h^n)$. Moreover, we have chosen the domain $\Omega=[0,2]^2$ using a structured mesh, and all the simulations are carried out using $\textbf{FreeFem++}$ software.  We will also compare with the scheme \textbf{US} studied in \cite{FMD2}. We use Newton's method to approach the nonlinear schemes \textbf{US} and \textbf{UV}; while for the scheme \textbf{US}$_\varepsilon$, we use the following Picard method:\\

\begin{itemize}
	\item \underline{Picard method to approach a solution $(u^{n}_\varepsilon,{\boldsymbol{\sigma}}^{n}_\varepsilon)$ of the scheme \textbf{US}$_\varepsilon$}:\\
	{\bf Initialization ($l=0$):} Set $(u^{0}_\varepsilon,{\boldsymbol{\sigma}}^{0}_\varepsilon)=(u^{n-1}_\varepsilon,{\boldsymbol{\sigma}}^{n-1}_\varepsilon)\in  U_h\times {\boldsymbol{\Sigma}}_h$.\\
	{\bf Algorithm:} Given $(u^{l}_\varepsilon,{\boldsymbol{\sigma}}^{l}_\varepsilon)\in  U_h\times {\boldsymbol{\Sigma}}_h$, compute $(u^{l+1}_\varepsilon,{\boldsymbol{\sigma}}^{l+1}_\varepsilon)\in  U_h\times {\boldsymbol{\Sigma}}_h$ such that
	$$
	\left\{
	\begin{array}
	[c]{lll}%
	\frac{1}{k}(u^{l+1}_\varepsilon,\bar{u})^h + (\nabla u^{l+1}_\varepsilon,\nabla \bar{u}) = \frac{1}{k}(u^{n-1}_\varepsilon,\bar{u})^h -(\Lambda_\varepsilon (u^{l}_\varepsilon){\boldsymbol{\sigma}}^{l}_\varepsilon,\nabla \bar{u}), \ \ \forall \bar{u}\in U_h,\\
	\frac{1}{k}({\boldsymbol \sigma}^{l+1}_\varepsilon,\bar{\boldsymbol \sigma}) + (B{\boldsymbol \sigma}^{l+1}_\varepsilon,\bar{\boldsymbol \sigma}) =\frac{1}{k}({\boldsymbol \sigma}^{n-1}_\varepsilon,\bar{\boldsymbol \sigma}) +
	(\Lambda_\varepsilon (u^{l+1}_\varepsilon) \nabla  u^{l+1}_\varepsilon,\bar{\boldsymbol \sigma}),\ \ \forall
	\bar{\boldsymbol \sigma}\in \Sigma_h,
	\end{array}
	\right.  
	$$
	until the stopping criterion $\max\left\{\displaystyle\frac{\Vert
		u^{l+1} - u^{l}\Vert_{0}}{\Vert u^{l}\Vert_{0}},\displaystyle\frac{\Vert
		{\boldsymbol{\sigma}}^{l+1} - {\boldsymbol{\sigma}}^{l}\Vert_{0}}{\Vert
		{\boldsymbol{\sigma}}^{l}\Vert_{0}}\right\}\leq tol$. 
\end{itemize}
In all the cases, we consider $tol=10^{-4}$.

\subsection{Positivity}\label{SimPos}
The aim of this subsection is to compare the fully discrete schemes $\textbf{UV}$, $\textbf{US}$ and \textbf{US}$_\varepsilon$ in terms of positivity. 
Theoretically,  for all schemes, is not clear the positivity of the variable $u^n_h$. However, for the scheme \textbf{US}$_\varepsilon$, it was proved that $\Pi^h(u^n_{\varepsilon-})\rightarrow 0$ in $L^2(\Omega)$ as $\varepsilon\rightarrow 0$ (see Theorem \ref{AAPP}). For this reason, in Figure \ref{fig:PosiU1} we compare the positivity of the variable $u^n$ in the schemes, taking the spatial parameter $h=1/20$,  a small time step $k = 10^{-5}$ (in order to see the differences in the spatial approximations), and the initial conditions (see Figure \ref{fig:initcond1}): 
$$u_0\!\!=\!\!-10xy(2-x)(2-y)exp(-10(y-1)^2-10(x-1)^2)+10.0001$$
and 
$$v_0\!\!=\!\!100xy(2-x)(2-y)exp(-30(y-1)^2-30(x-1)^2)+0.0001.$$
\begin{figure}[h]
	\begin{center}
		{\includegraphics[height=0.4\linewidth]{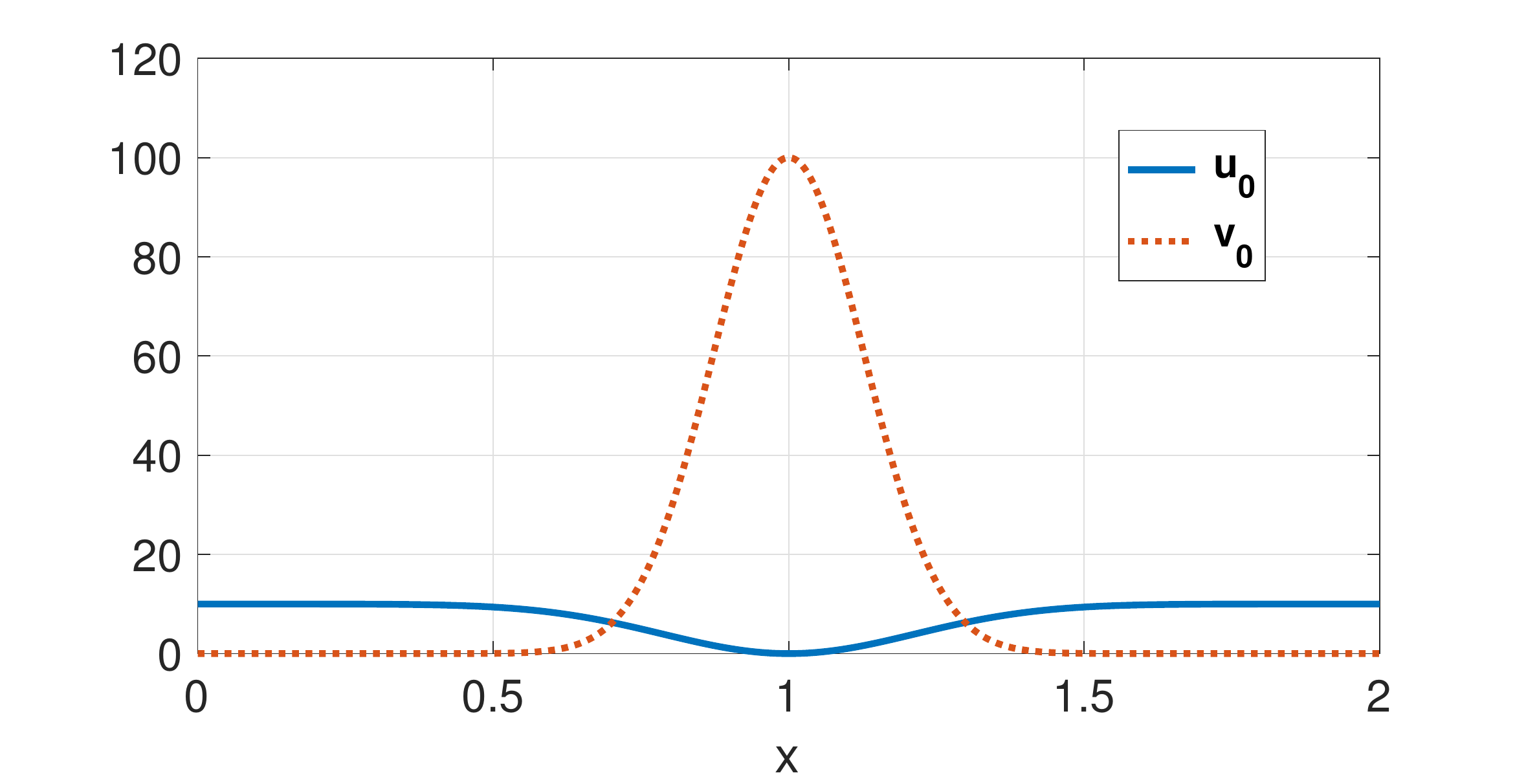}}
		\caption{Cross section at $y=1$ of the initial cell density $u_0$ and chemical concentration $v_0$.
			\label{fig:initcond1}}
	\end{center}
\end{figure}
Note that $u_0,v_0>0$ in $\Omega$, $\min (u_0)=u_0(1,1)=0.0001$ and $\max (v_0)=v_0(1,1)=100.0001$. We obtain that (see Figure \ref{fig:PosiU1}):
\begin{enumerate}
	\item In all schemes, the discrete cell density $u_h^n$ takes negative values for some $\x \in \Omega$ in some times $t_n>0$.
	\item In the scheme \textbf{US}$_\varepsilon$, the negative values of $u^n_{\varepsilon}$ are closer to $0$ as $\varepsilon \rightarrow 0$.
	\item The scheme \textbf{US}$_\varepsilon$ evidence ``better positivity'' than the schemes \textbf{UV} and \textbf{US}, because the ``greater'' negative values for \textbf{US}$_\varepsilon$ are of order $10^{-2}$, while the another schemes reach values greater than $-1$.
\end{enumerate}

\begin{figure}[htbp]
	\centering 
	\subfigure[Scheme \textbf{US}$_\varepsilon$]{\includegraphics[width=78mm]{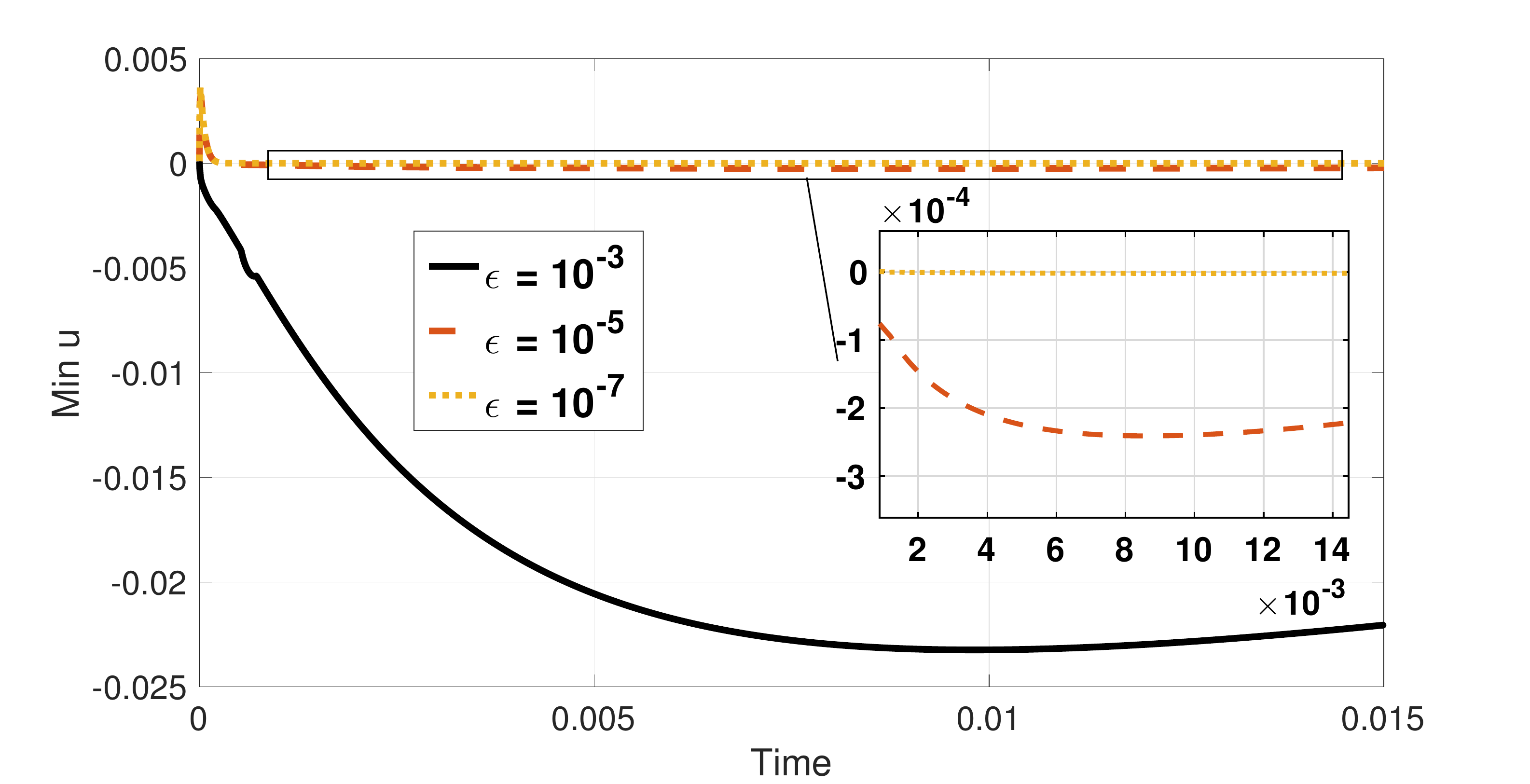}} \hspace{0.1 cm} 
	\subfigure[Schemes \textbf{UV} and \textbf{US}]{\includegraphics[width=78mm]{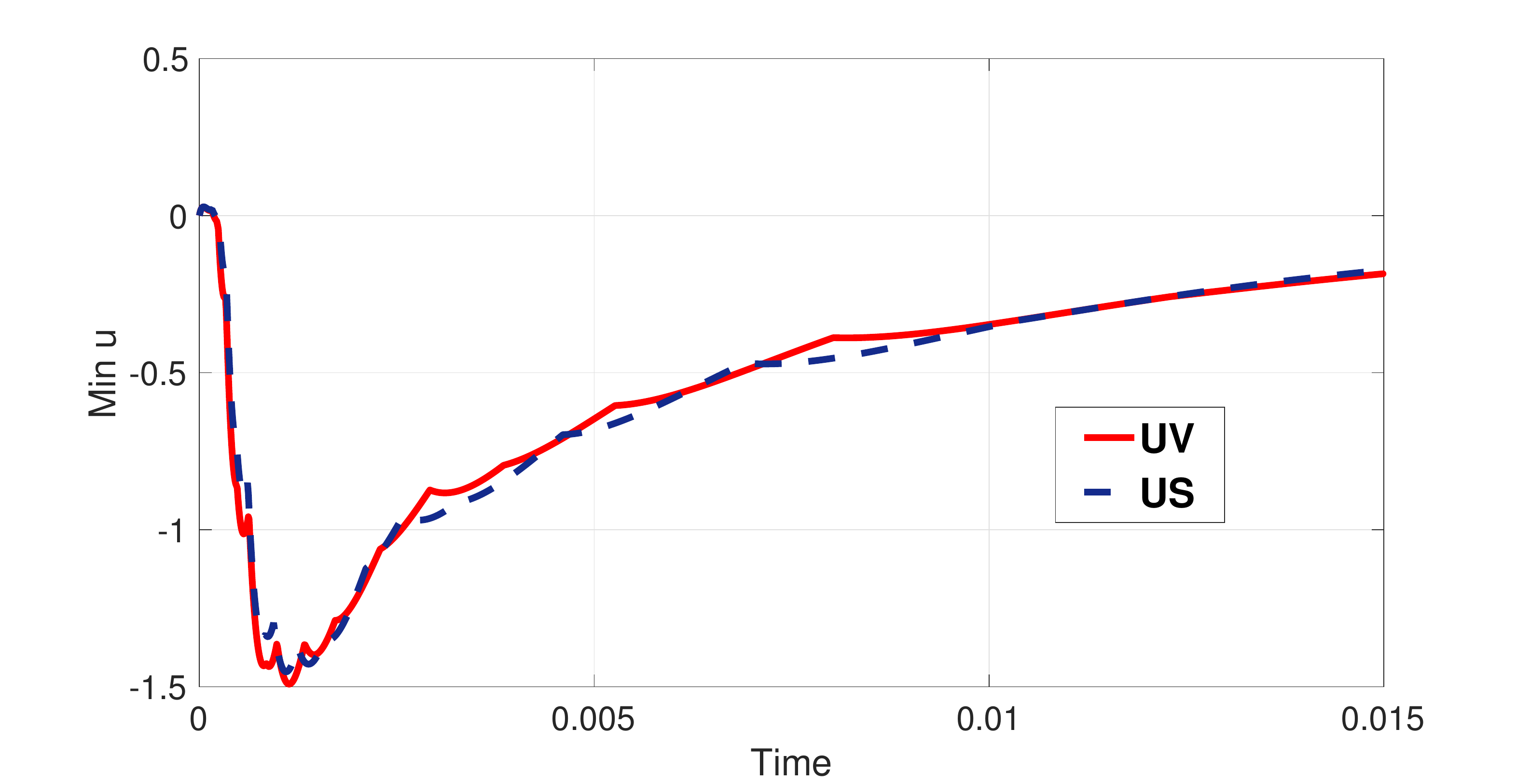}}
	\caption{Minimum values of $u^n_h$} \label{fig:PosiU1}
\end{figure}

\subsection{$\varepsilon$-aproximated positivity vs spurious oscillations}
In this subsection, we present some numerical experiments relating the results concerning to the negativity of the discrete cell density observed in Subsection \ref{SimPos} with the spurious oscillations that could appear. With this aim,  we consider $k=10^{-5}$, $h=\frac{1}{25}$, $\varepsilon=10^{-6}$ (for the scheme \textbf{US}$_\varepsilon$) and the following initial conditions:
$$u_0=5cos(2\pi x)cos(2\pi y)+5.0001 \ \ \mbox{and} \ \ 
v_0=-170cos(2 \pi x)cos(2 \pi y))+170.0001$$
in which, the places with the highest chemical concentration have lower cell density, in order to force to the cell density  to be very close to zero. Note that $u_0,v_0>0$ in $\Omega$, $\min (u_0)=u_0(1,1)=0.0001$ and $\max (v_0)=v_0(1,1)=170.0001$. 

\

We observe that, in the case of the schemes \textbf{UV} and \textbf{US}, some spuriuos oscillations appear when the discrete cell density takes negative values (which makes simulations unreliable in this ``extreme'' case); while, in the case of the scheme \textbf{US}$_\varepsilon$, the $\varepsilon$-aproximated positivity favors the non-appearance of spurious oscillations (see Figure \ref{fig:NC1}).

\begin{minipage}{\textwidth}
	\begin{tabular}{p{1.15cm} p{4.3cm}p{4.3cm} p{4.3cm}} 
		Time & \hspace{0.9 cm} Scheme \textbf{UV}&\hspace{0.9 cm}	Scheme \textbf{US} & \hspace{0.9 cm} 	Scheme \textbf{US}$_\varepsilon$\\[1mm]
		\scriptsize{t=0}&	\includegraphics[width=44mm]{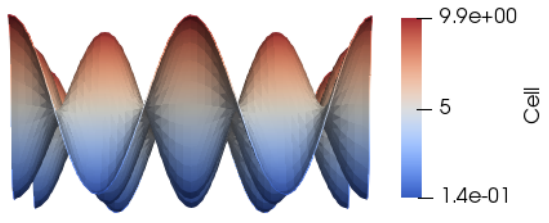} & \hspace{-2mm} \includegraphics[width=44mm]{UV1} & \hspace{-2mm} \includegraphics[width=44mm]{UV1} \\[2mm]
		%	\multicolumn{3}{|c|}{t=0}
		% t=0 &   t=0 &   t=0 
		%\\[2mm]
		\scriptsize{t=7e-5}&	\includegraphics[width=44mm]{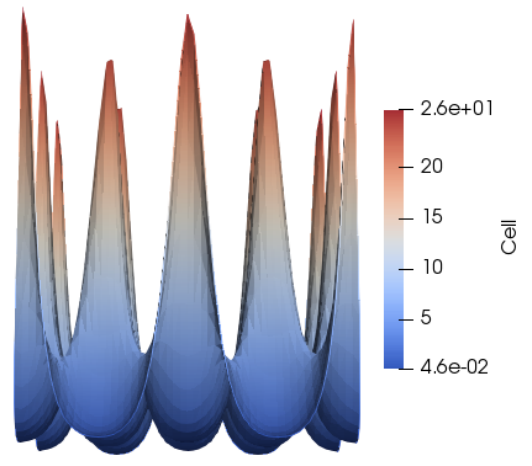} & \includegraphics[width=43mm]{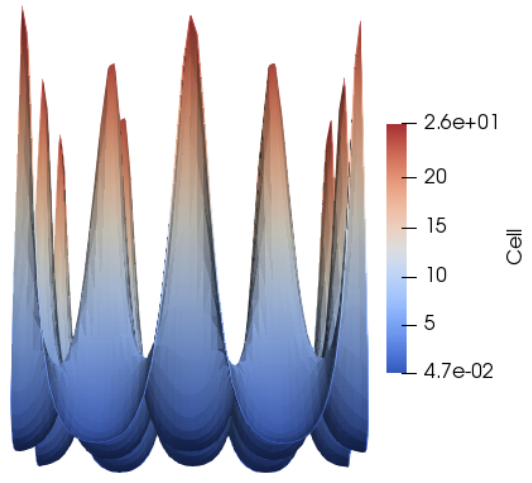} & \includegraphics[width=44mm]{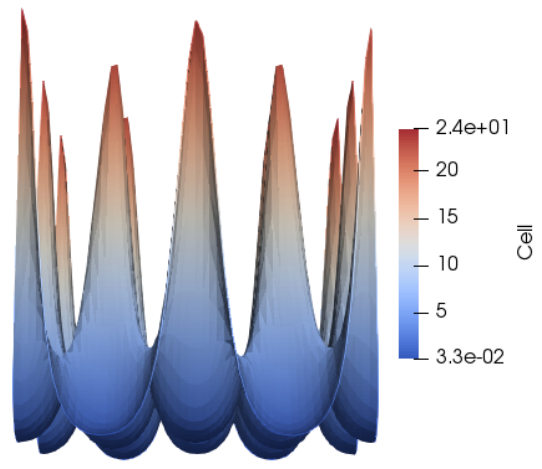}  
		\\[2mm]
		%(c)  t=12e-5 &  (d)  t=12e-5 & (d)  t=12e-5 \\[2mm]
		\scriptsize{t=2.4e-4}&	\includegraphics[width=44mm]{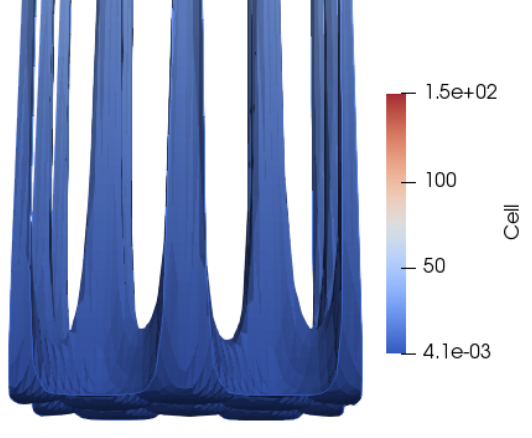} & \includegraphics[width=44mm]{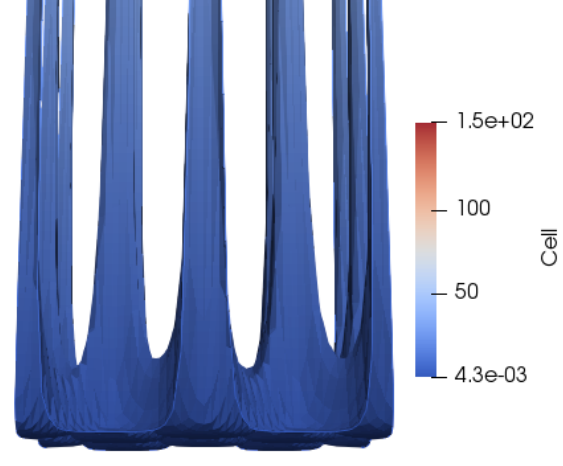} & \includegraphics[width=44mm]{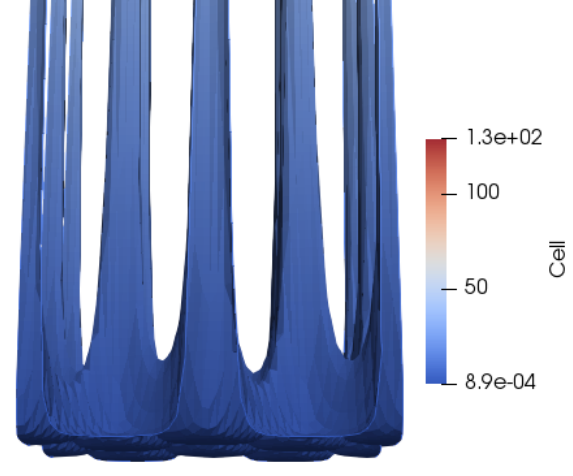}  
		\\[2mm]
		%(e) t=30e-5 & (f)  t=30e-5  & (f)  t=30e-5 
		%\\[2mm]
		\scriptsize{t=4.4e-4}&	\includegraphics[width=44mm]{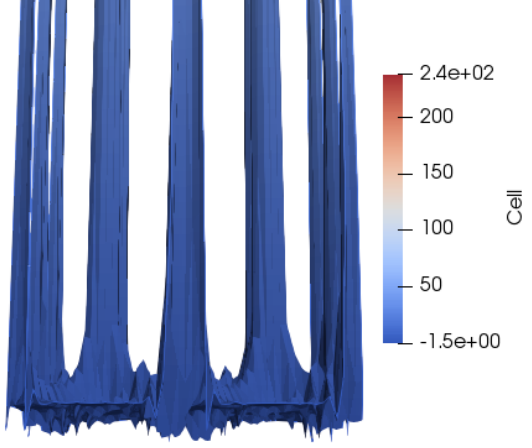} & \includegraphics[width=44mm]{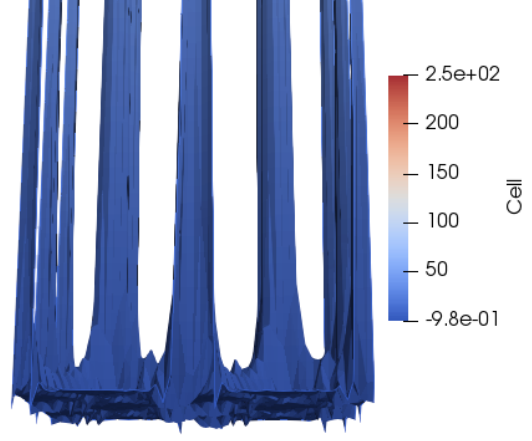} & \includegraphics[width=44mm]{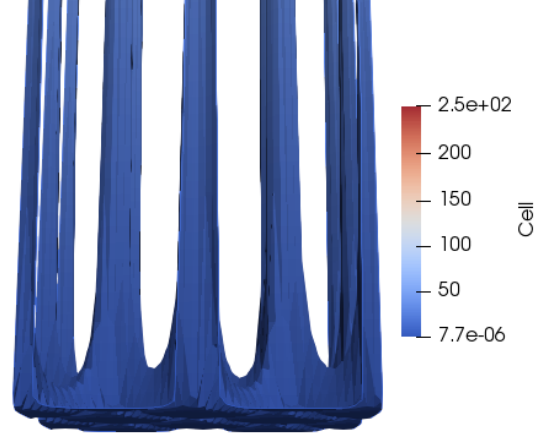} 
		\\[2mm]
		%(e) t=30e-5 & (f)  t=30e-5  & (f)  t=30e-5
		%\\[2mm]
		\scriptsize{t=6.9e-4} & \includegraphics[width=44mm]{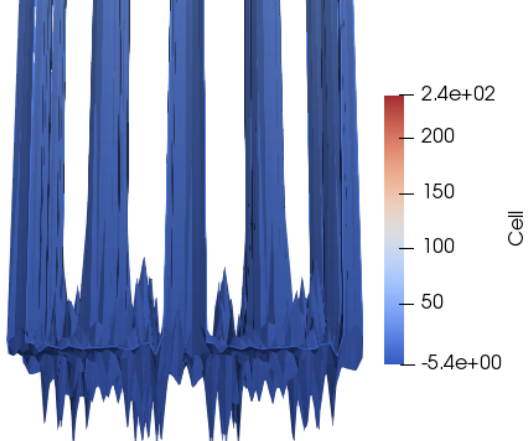} & \includegraphics[width=44mm]{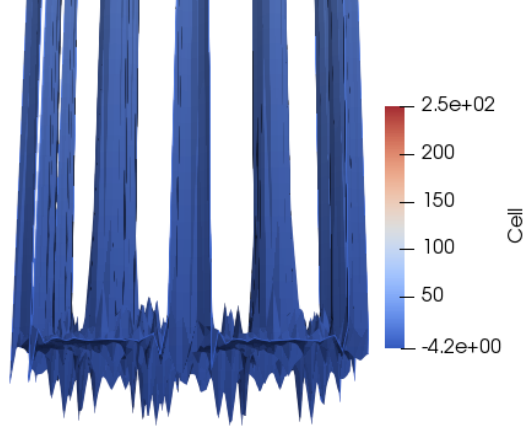} & \includegraphics[width=44mm]{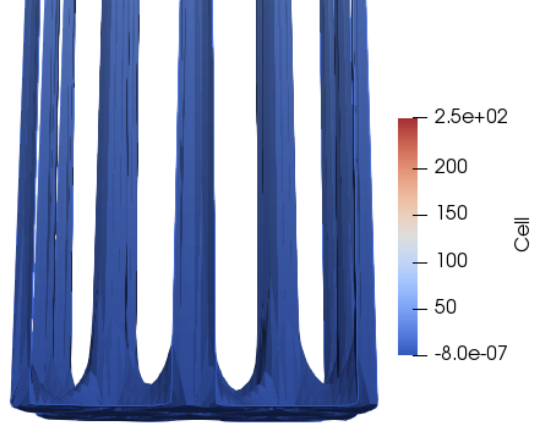}  
		\\[-2mm]
		%	(e) t=30e-5 & (f)  t=30e-5  & (f)  t=30e-5
		%\\[-2mm]
	\end{tabular}
	\figcaption{Positivity vs spurious oscillations of the discrete cell density at different times.} \label{fig:NC1}
\end{minipage}

\subsection{Energy-Stability}
Previously, it was proved that the scheme \textbf{UV} is unconditionally energy-stable with respect to
 the energy $\mathcal{E}(u,v)$ given in (\ref{eneruva}) (in the primitive variables $(u,v)$), while the schemes
\textbf{US} and \textbf{US}$_\varepsilon$ are unconditionally energy-stables with respect to the modified energy $\widetilde{\mathcal{E}}(u,{\boldsymbol \sigma})$ given in (\ref{nueva-2}). In this section, we compare numerically the energy stability of the schemes with respect to the ``exact'' energy $\mathcal{E}(u,v)$ which comes from the continuous problem, and to study the behaviour of the corresponding discrete residual of the energy law (\ref{wsd}):
\begin{equation*}
RE^n:=\delta_t \mathcal{E}(u^n_h,{v}^n_h)+ \Vert \nabla u^n_h\Vert_{0}^{2}
+
\displaystyle\frac{1}{2}\Vert (A_h-I) v^n_h\Vert_{0}^{2} +
\displaystyle\frac{1}{2}\Vert \nabla v^n_h\Vert_{0}^{2}.
\end{equation*}
With this aim, we consider the parameters $k=10^{-4}$, $h=\frac{1}{30}$ and the initial conditions 
$$u_0\!\!=\!\!-10xy(2-x)(2-y)exp(-10(y-1)^2-10(x-1)^2)+10.0001$$
and 
$$v_0\!\!=\!\!20xy(2-x)(2-y)exp(-30(y-1)^2-30(x-1)^2)+0.0001,$$
obtaining that:
\begin{enumerate}
\item[(a)] {
All schemes satisfy the energy decreasing in time property for the energy $\mathcal{E}(u,v)$, that is, $\mathcal{E}(u^n_h,v^n_h)\le \mathcal{E}(u^{n-1}_h,v^{n-1}_h)$ for all $n$, see Figure \ref{fig:EnE}(a).}
\item[(b)]{The schemes \textbf{UV} and \textbf{US} satisfy the discrete energy law $RE^n \leq 0$ for all $n\geq 1$; while the scheme \textbf{US}$_\varepsilon$ evidence positive values for $RE^n$ for some $n\geq 1$, but these values are very close to $0$  (see Figure \ref{fig:EnE})(b).}

\end{enumerate}
\begin{figure}[htbp]
	\centering 
	\subfigure[Energy $\mathcal{E}(u^n_h,{v}^n_h)$]{\includegraphics[width=78mm]{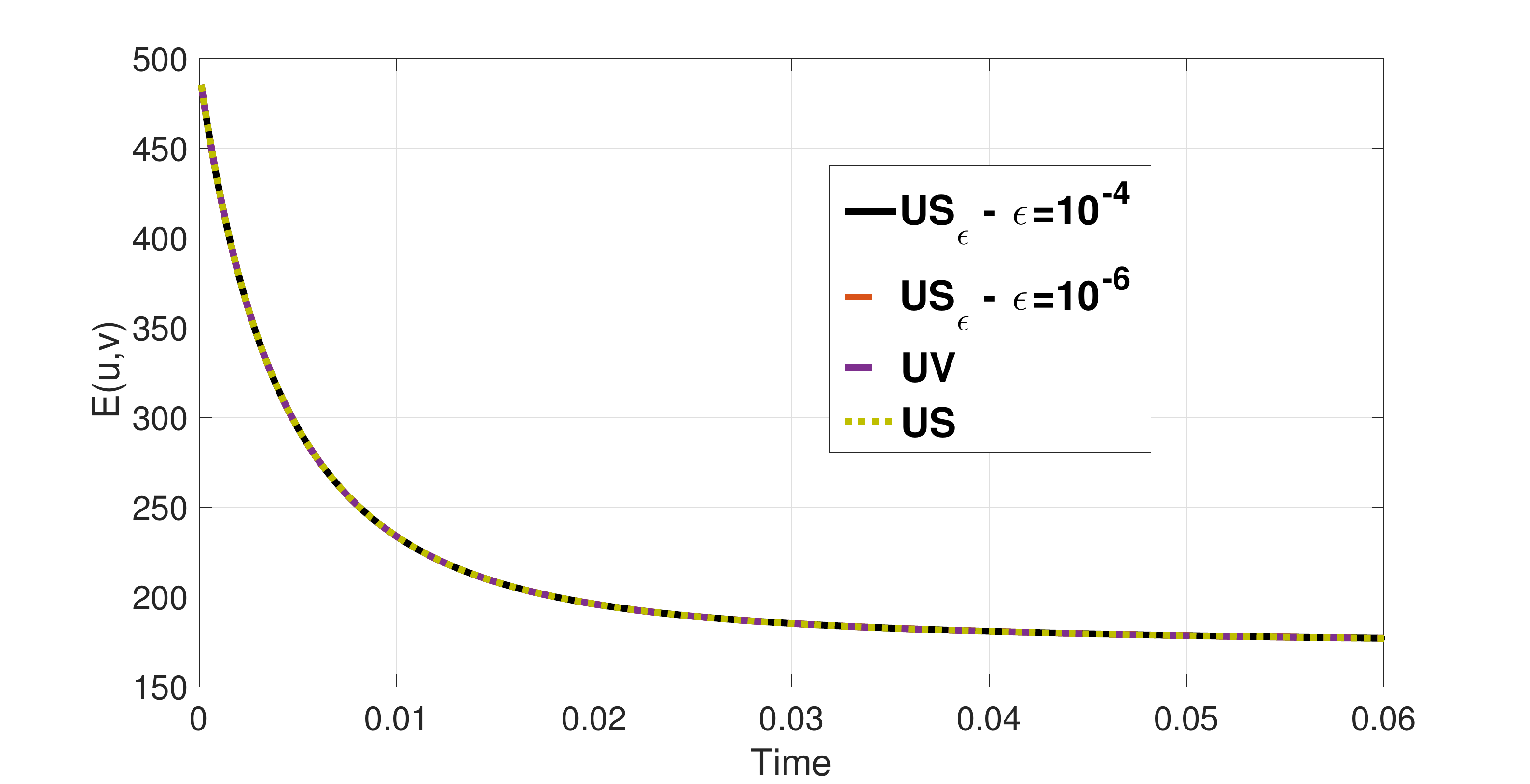}} \hspace{0,1 cm} 
	\subfigure[Discrete residual $RE^n$]{\includegraphics[width=78mm]{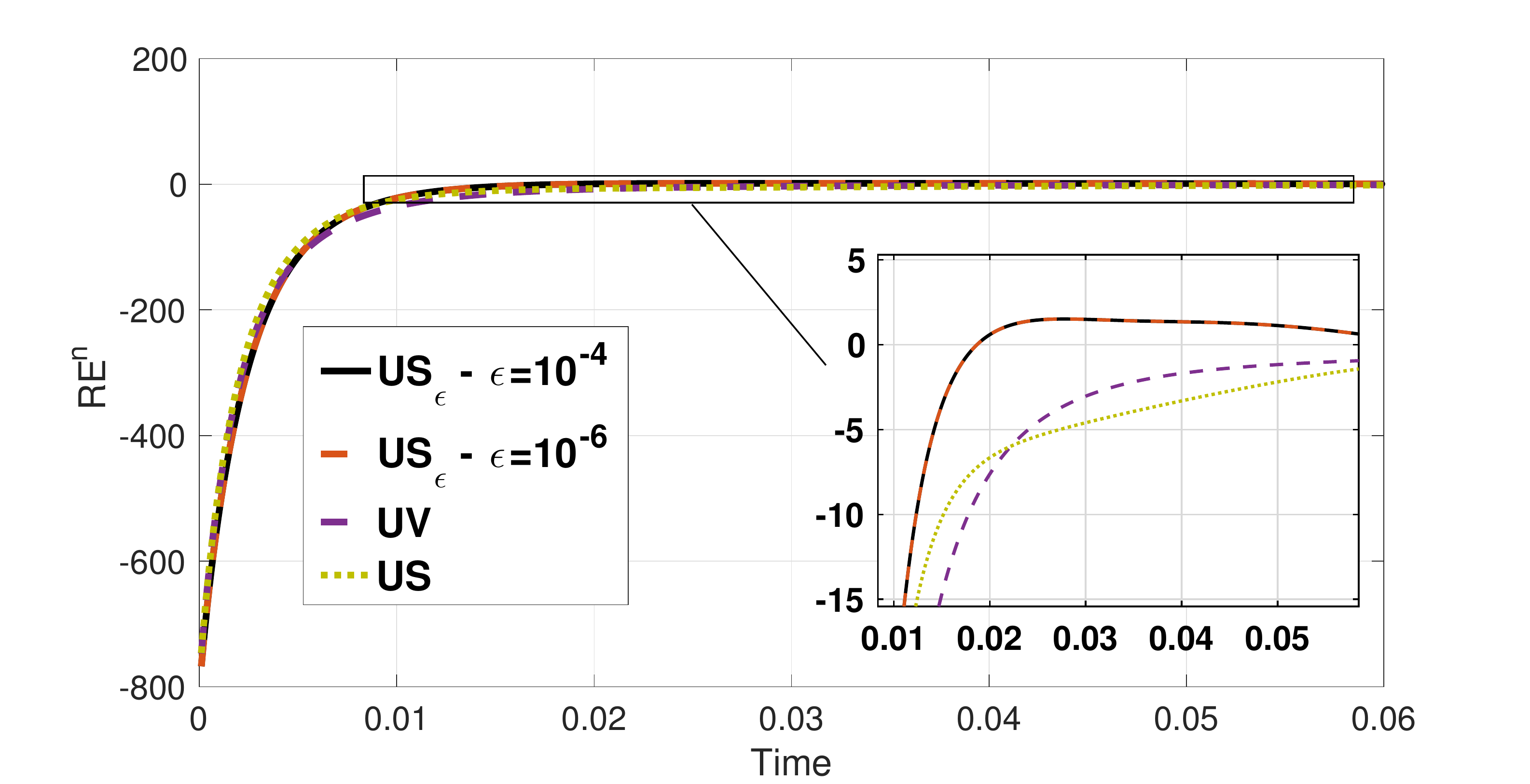}}
	\caption{Energy-stability of the schemes \textbf{UV}, \textbf{US} and \textbf{US}$_\varepsilon$.} \label{fig:EnE}
\end{figure}

\subsection{Asymptotic behaviour}
In this subsection, we present some numerical experiments in order to illustrate the large-time behavior of approximated solutions computed by using the schemes \textbf{UV}, \textbf{US} and \textbf{US}$_\varepsilon$ in two different situations. In the first test, we consider the initial conditions such that the places with the highest chemical concentration have lower cell density; while in the second test, the places with the highest initial chemical concentration have the highest initial cell density. In both situations, we consider $k=10^{-3}$ and $h=\frac{1}{25}$. Moreover, for the scheme \textbf{US}$_\varepsilon$, we consider $\varepsilon=10^{-5}$.

\begin{itemize}
	\item Test 1: We choose the initial conditions (see Figure \ref{fig:initcond3}): 
	$$u^1_0=5cos(2\pi x)cos(2\pi y)+5.0001 \ \ \mbox{and} \ \ 
	v^1_0=-15cos(2 \pi x)cos(2 \pi y))+24.$$
	
		\item Test 2: We choose the initial conditions: 
	$$u^2_0=u_0^1 \ \ \mbox{and} \ \ 
	v^2_0=15cos(2 \pi x)cos(2 \pi y))+24.$$
\end{itemize}
 
\begin{figure}[h]
	\begin{center}
		{\includegraphics[height=0.35\linewidth]{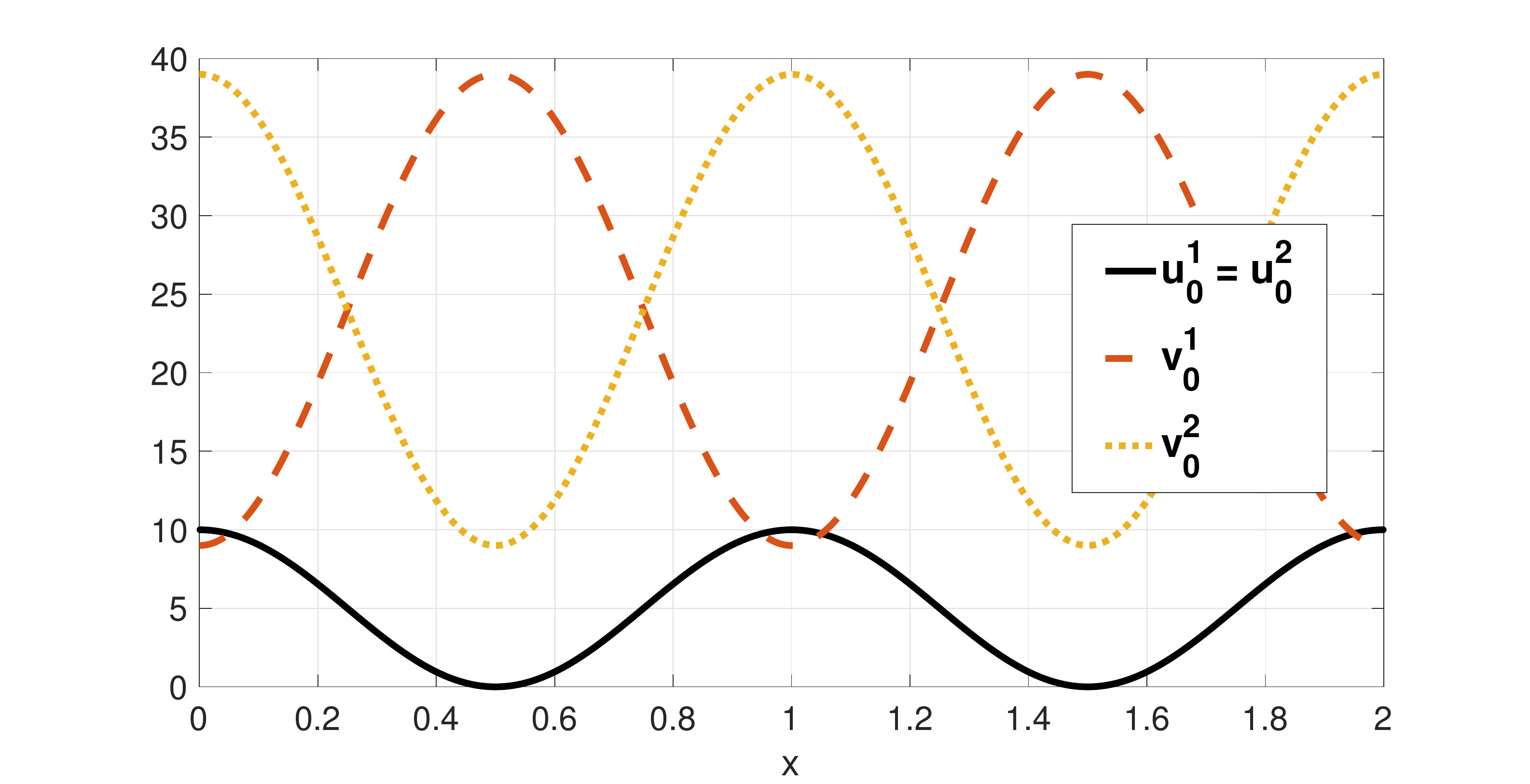}}
		\caption{Cross section at $y=1$ of the initial cell densities $u_0^1=u_0^2$ and chemical concentrations $v_0^1$, $v_0^2$.
			\label{fig:initcond3}}
	\end{center}
\end{figure}

In both cases, we observe that $\Vert (u^n_h - m_0,\nabla v^n_h) \Vert_0^2$ 
%and $\Vert u^n_h - m_0 \Vert_1^2$ 
decreases to $0$ faster than $\Vert v^n_h - (m_0)^2 \Vert_0^2$. 
%and $\Vert A_h(v^n_h - (m_0)^2) \Vert_0^2$.
 In Figures \ref{fig:T1u}-\ref{fig:T2u} we observe an  exponential decay (at least)  of $(u^n_h,v^n_h)$  to $(m_0,(m_0)^2)$. These facts  are in agreement with the theoretical results proved in this paper.

\begin{figure}[htbp]
	\centering 
{\includegraphics[width=78mm]{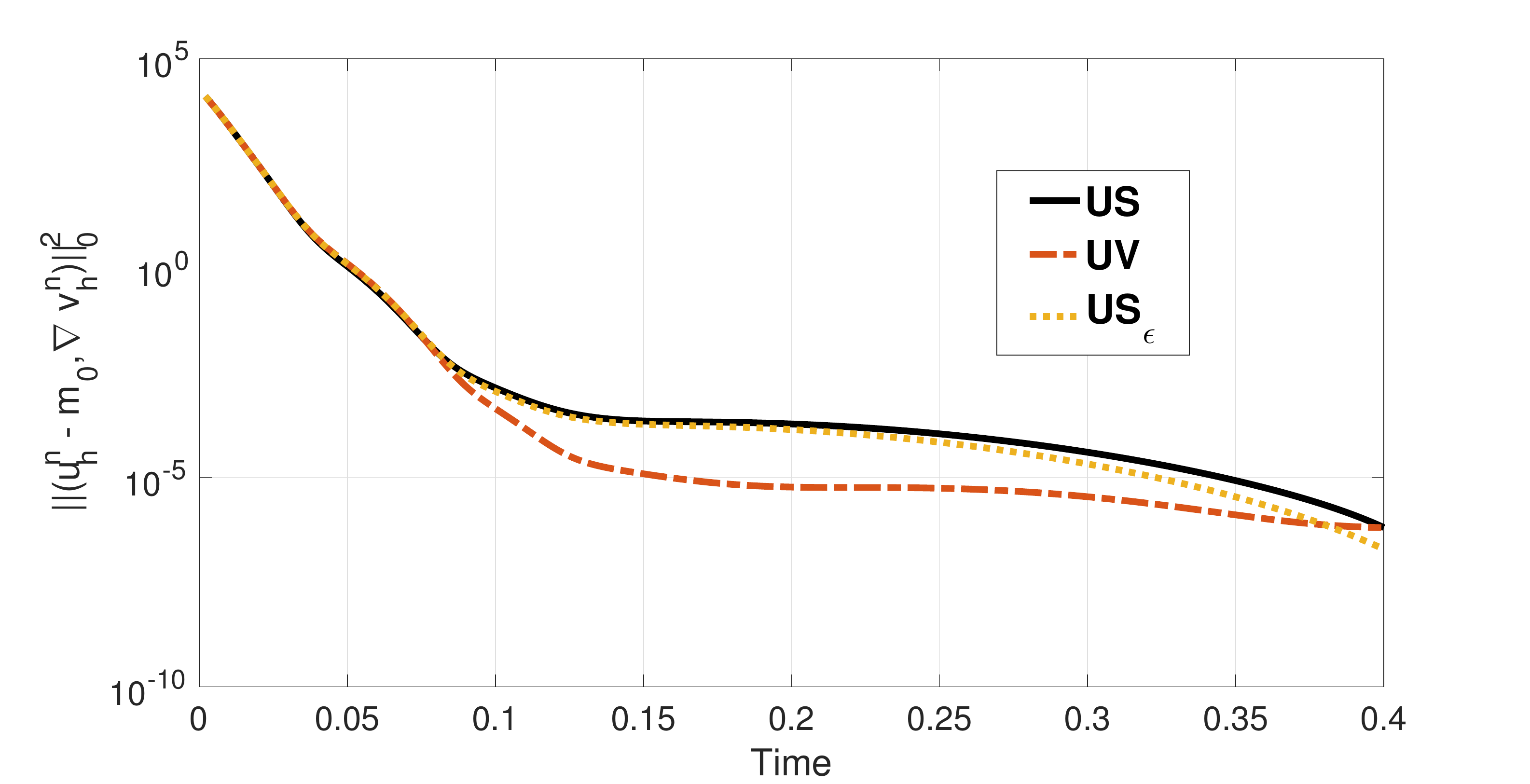}} \hspace{0,1 cm} 
%{\includegraphics[width=78mm]{StrUT1}}\hspace{0,1 cm}
	{\includegraphics[width=78mm]{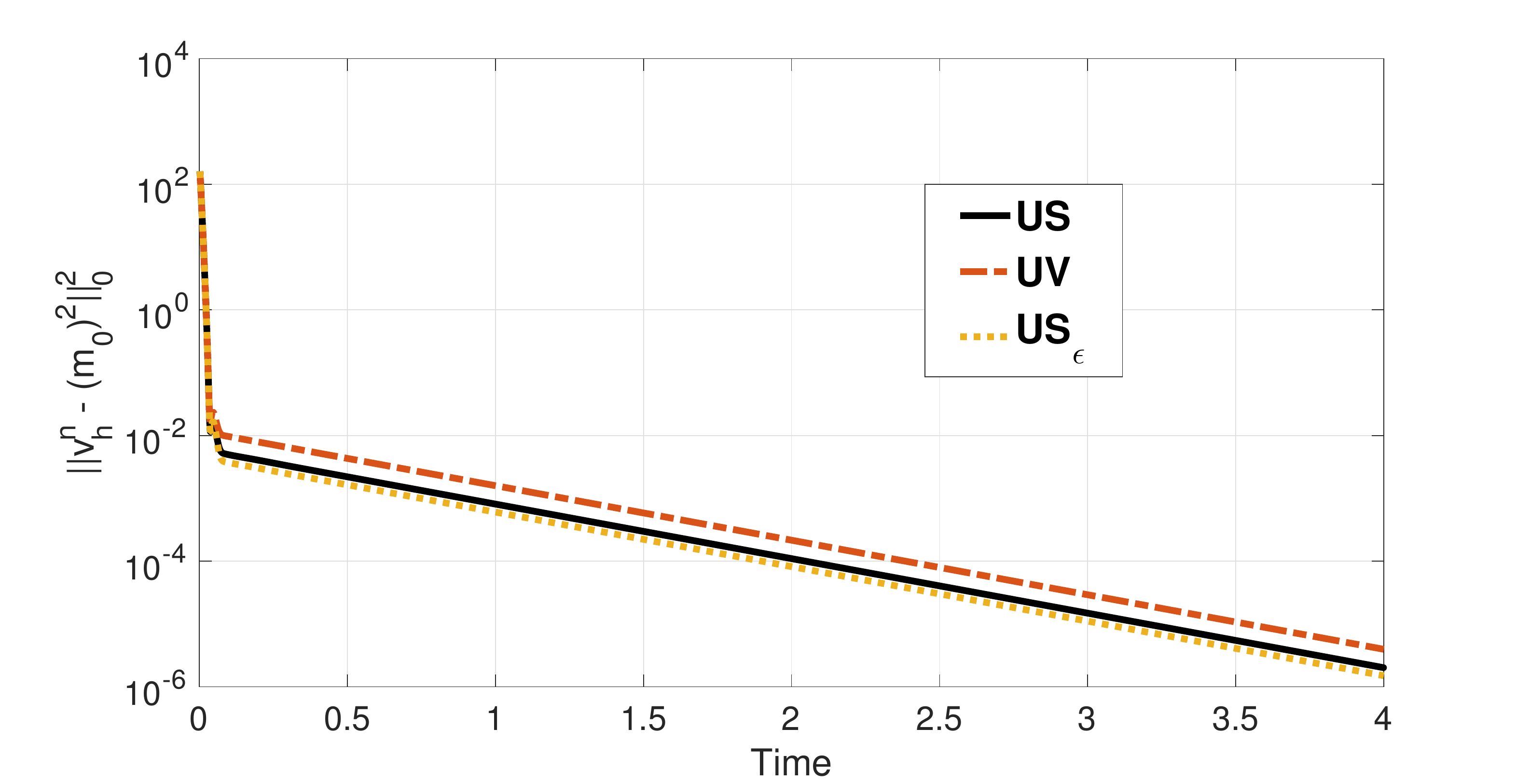}} \hspace{0,1 cm} 
%{\includegraphics[width=78mm]{StrVT1}}  
	\caption{Evolution of $\Vert (u^n_h - m_0,\nabla v^n_h) \Vert_0^2$
		%, $\Vert u^n_h - m_0 \Vert_1^2$, 
		and $\Vert v^n_h - (m_0)^2 \Vert_0^2$ 
		%and $\Vert A_h(v^n_h - (m_0)^2) \Vert_0^2$
		 in test 1.} \label{fig:T1u}
\end{figure}

\begin{figure}[htbp]
	\centering 
	{\includegraphics[width=78mm]{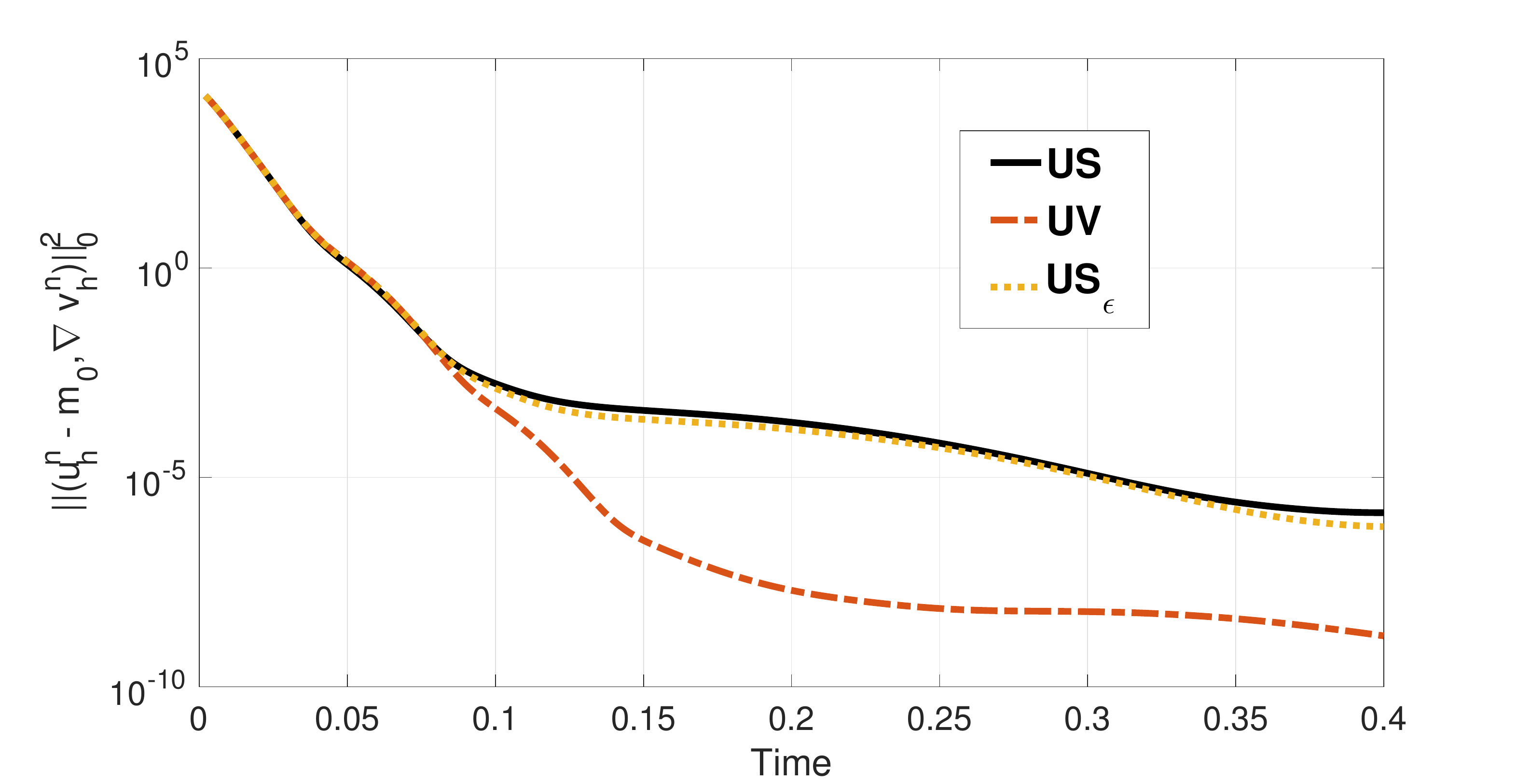}} \hspace{0,1 cm} 
	%{\includegraphics[width=78mm]{StrUT2}}	\hspace{0,1 cm}
	{\includegraphics[width=78mm]{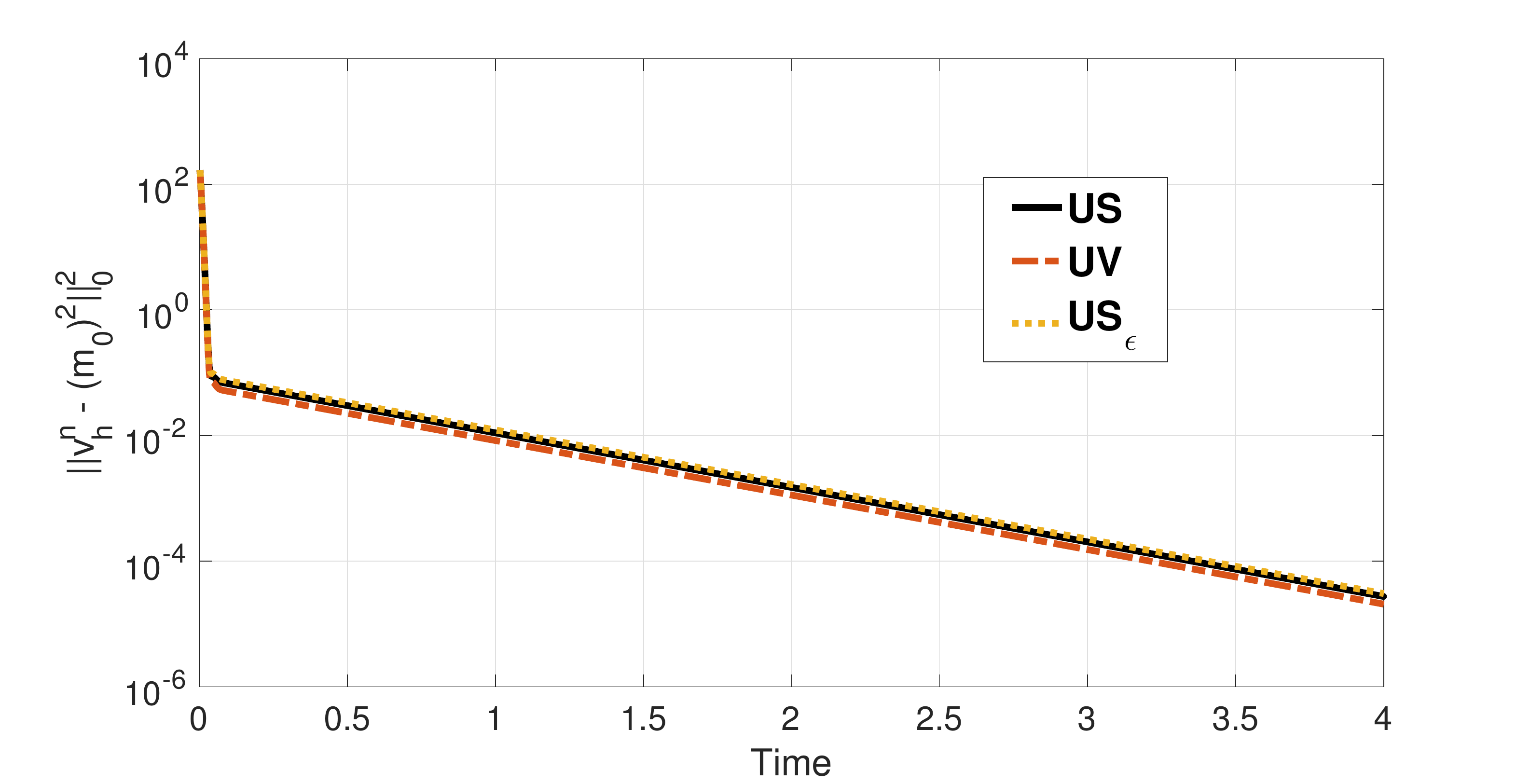}} \hspace{0,1 cm} 
	%{\includegraphics[width=78mm]{StrVT2}}  
	\caption{Evolution of $\Vert (u^n_h - m_0,\nabla v^n_h) \Vert_0^2$
		%, $\Vert u^n_h - m_0 \Vert_1^2$, 
		and $\Vert v^n_h - (m_0)^2 \Vert_0^2$ 
		%and $\Vert A_h(v^n_h - (m_0)^2) \Vert_0^2$ 
		in test 2.} \label{fig:T2u}
\end{figure}

\section{Conclusions}\label{Con}
In this paper, we study two fully discrete FE schemes for a repulsive chemotaxis model with quadratic signal production, called \textbf{UV} (the FE backward Euler in variables $(u,v)$) and \textbf{US}$_\varepsilon$ (obtained by mixing the scheme \textbf{US} proposed in\cite{FMD2} with a regularization technique). For these numerical schemes we obtain better properties than proved for the scheme \textbf{US} in \cite{FMD2}. Specifically, the comparison between the numerical schemes \textbf{UV} and \textbf{US}$_\varepsilon$, and the scheme \textbf{US}, allows us to conclude that, from the theoretical point of view:
\begin{enumerate}
	\item\label{one} By imposing the ``compatibility'' condition  $(\mathbb{P}_m,\mathbb{P}_{2m})$-continuous FE (with $m\geq 1$) for $(u,v)$,  the scheme \textbf{UV} is energy-stable (in the primitive variables $(u,v)$). In the case of the schemes \textbf{US} and \textbf{US}$_\varepsilon$, it can be obtained energy-stability but with respect to a modified energy written in terms of $(u,{\boldsymbol{\sigma}})$. 
	\item As a consequence of item \ref{one}, the exponential convergence of the scheme \textbf{UV} to the constant states $m_0$ and $(m_0)^2$ (when the time goes to infinity) can be proved in weak norms for $u$ and strong norms for $v$ (equal than the continuous case); while in the schemes \textbf{US} and \textbf{US}$_\varepsilon$, can be proved also exponential convergence towards $m_0$ and $(m_0)^2$, but only in weak norms for $u$ and $v$.
	\item Aproximate positivity for the discrete solutions is proved for the scheme \textbf{US}$_\varepsilon$, but it is not clear how to prove neither positivity nor approximated positivity for the schemes \textbf{US} and \textbf{UV}.
\end{enumerate}   
From the numerical point of view, we have obtained that:
\begin{enumerate}
	\item The scheme \textbf{US}$_\varepsilon$ evidence ``better positivity'' than the schemes \textbf{UV} and \textbf{US}. Moreover, for the scheme \textbf{US$_\varepsilon$} it was observed numerically that $\underset{\overline{\Omega}\times[0,T]}{\min}\ u^n_{\varepsilon} \rightarrow 0$ as $\varepsilon\rightarrow 0$.
	\item In some cases, for example when negative values are obtained for $u_h$, some spurious oscillations are observed in the schemes \textbf{UV} and \textbf{US}; while in the scheme \textbf{US}$_\varepsilon$, the approximated positivity of $u_h$ favors  the non-appearance of spurious oscillations.
	\item The three schemes have decreasing in time energy $\mathcal{E}(u,v)$.
	\item  It is observed, for the three schemes, an  exponential decay (at least)  of $(u^n_h,v^n_h)$ in weak-strong norm to $(m_0,(m_0)^2)$.
\end{enumerate}

\section*{Acknowledgements}
The authors have been partially supported by MINECO grant MTM2015-69875-P
(Ministerio de Econom\'{\i}a y Competitividad, Spain) with the participation of FEDER. The first and second authors have also been supported by PGC2018-098308-B-I00 (MCI/AEI/FEDER, UE); and  
the third author has also been supported by Vicerrector\'ia de Investigaci\'on y Extensi\'on of Universidad Industrial de Santander.

\end{document}